\documentclass{amsart}

\usepackage{amssymb,amsmath,amsthm,amsfonts}
\usepackage[mathscr]{euscript}
\usepackage{dsfont}
\usepackage{graphicx}
\usepackage{mathrsfs}
\usepackage{esint}
\usepackage{xcolor}

\renewcommand{\leq}{\leqslant}
\renewcommand{\le}{\leqslant}
\renewcommand{\geq}{\geqslant}

\usepackage{color}
\definecolor{citation}{rgb}{0.2,0.5,0.2}
\definecolor{formula}{rgb}{0.1,0.2,0.5}
\definecolor{url}{rgb}{0,0.2,0.7}
\usepackage[colorlinks=true,linkcolor=formula,urlcolor=url,citecolor=citation]{hyperref}
\allowdisplaybreaks[4]

\newtheorem{theorem}{Theorem}[section]
\newtheorem{corollary}[theorem]{Corollary}
\newtheorem{lemma}[theorem]{Lemma}
\newtheorem{prop}[theorem]{Proposition}

\theoremstyle{definition}
\newtheorem{defn}[theorem]{Definition}
\theoremstyle{remark}
\newtheorem{rem}[theorem]{Remark}
\theoremstyle{remark}
\newtheorem{ex}[theorem]{Example}
\numberwithin{equation}{section}

\def\R {\mathbb{R}}

\def\N {\mathbb{N}}
\def\S {\mathbb{S}}
\def\Z {\mathbb{Z}}

\def\o {\omega}

\def\eps{\varepsilon}

\newlength{\defbaselineskip}
\newcommand{\setlinespacing}[1]
           {\setlength{\baselineskip}{#1 \defbaselineskip}}

\DeclareGraphicsExtensions{.eps,.pdf,.png}

\begin{document}

\title[]
{A Bourgain-Brezis-Mironescu \\
characterization of higher order \\
Besov-Nikol'skii spaces}
%Monotonicity results for minimizers of fully nonlinear elliptic systems of equations???

\author[Julien Brasseur]{Julien Brasseur}
\address[Julien Brasseur]{INRA Avignon, unit\'e BioSP and Aix-Marseille Univ, CNRS,
Centrale Marseille,
I2M, Marseille,
France}
\email{julien.brasseur@univ-amu.fr, julien.brasseur@inra.fr}

\begin{abstract}
We study a class of nonlocal functionals in the spirit of the recent characterization of the Sobolev spaces $W^{1,p}$ derived by Bourgain, Brezis and Mironescu. We show that it provides a common roof to the description of the $BV(\R^N)$, $W^{1,p}(\R^N)$, $B_{p,\infty}^s(\R^N)$ and $C^{0,1}(\R^N)$ scales and we obtain new equivalent characterizations for these spaces. We also establish a non-compactness result for sequences and new (non-)limiting embeddings between Lipschitz and Besov spaces which extend the previous known results.
\end{abstract}



\maketitle

\tableofcontents

\section{Introduction}

\subsection{A brief state of art}
Let $(\rho_\varepsilon)_{\varepsilon>0}\subset L^1(\mathbb{R}^N)$ be a sequence of mollifiers, i.e.
a sequence of functions satisfying
\begin{align}
\left\{
\begin{array}{l}
\rho_\varepsilon\geq0~~\text{in}~~\mathbb{R}^N~~\text{for any}~~\varepsilon>0, \vspace{4pt} \\
%~~ \\
\displaystyle\int_{\mathbb{R}^N}\rho_\varepsilon(z)\mathrm{d}z=1~~\text{for any}~~\varepsilon>0, \vspace{3pt}\\
%~~ \\
\displaystyle\lim_{\varepsilon\downarrow0}\int_{|z|\geq\delta}\rho_\varepsilon(z)\mathrm{d}z=0~~\text{for all}~~\delta>0.
\end{array}
\right. \label{molli}
\end{align}
Let $M\in\N^*$, $1\le p< \infty$ and $s\in (0,M]$. We are interested in the properties of functions $f\in L^p(\mathbb{R}^N)$, satisfying
\begin{equation}\label{ener-eps}
\int_{\mathbb{R}^N}\rho_{\eps}(h)\,\o\left(\int_{\mathbb{R}^N}\frac{|\Delta_h^Mf(x)|^p}{|h|^{sp}} \mathrm{d}x\right)\mathrm{d}h\leq C \quad \text{ as }~~ \varepsilon\downarrow0,
\end{equation}
where $\omega:\R_+\to\R_+$ is an increasing, concave function
%\begin{align}
%m_1t^p\leq \omega(t)\leq m_2t^p,
%\end{align}
%for every $t\geq0$ with $m_1,m_2>0$,
and $\Delta_h^Mf(x)$ stands for the usual $M$-th order forward difference of $f$ given by
\begin{align}
\Delta_h^Mf(x):=\sum_{j=0}^M(-1)^{M-j}\binom{M}{j}f(x+hj), \quad x,h\in\R^N. \label{iterateddiff2}
\end{align}
The assumptions on $\o$ will be made precise later on. \\
%where for $k\in \N_0, \Delta_k^h(f)(x):=$ and $\o$ is a  increasing convex function such that

Functionals of the type of \eqref{ener-eps} were initially introduced by Bourgain, Brezis and Mironescu \cite{BBM,Brezis} to obtain a new characterization of the Sobolev space  $W^{1,p}(\mathbb{R}^N)$. Namely, for $M=s=1$ and $\o(t)=t$,
\eqref{ener-eps} reads
\begin{align}
\int_{\mathbb{R}^N}\int_{\mathbb{R}^N}\rho_\varepsilon(h)\frac{|f(x+h)-f(x)|^p}{|h|^{p}}\mathrm{d}x\mathrm{d}h\leq C \quad \text{ as }~~\varepsilon\downarrow0, \label{hypo4}
\end{align}
and the result of Bourgain, Brezis and Mironescu states that, any $f\in L^p(\mathbb{R}^N)$ satisfying \eqref{hypo4}
belongs to the Sobolev space $W^{1,p}(\mathbb{R}^N)$ if $1<p<\infty$, or to $BV(\mathbb{R}^N)$ if $p=1$, provided $(\rho_\eps)_{\eps>0}$ is radial. More precisely, they have shown that
\begin{align}
\lim_{\eps\downarrow0}\int_{\mathbb{R}^N}\int_{\mathbb{R}^N}\rho_\varepsilon(h)\frac{|f(x+h)-f(x)|^p}{|h|^{p}}\mathrm{d}x\mathrm{d}h=K_{p,N}\|\nabla f\|_{L^p(\R^N)}^p,
\end{align}
where
$$ K_{p,N}:=\int_{\S^{N-1}}|\sigma\cdot e|^p\mathrm{d}\mathcal{H}^{N-1}(\sigma), \quad e\in\S^{N-1}. $$
As a result, they were able to establish the following limiting embedding
\begin{align}
\lim_{r\uparrow 1}\,(1-r)p\|f\|_{W^{r,p}(\R^N)}^p=K_{p,N}\|\nabla f\|_{L^p(\R^N)}^p. \label{BBMlimimb}
\end{align}
Since this original work, numerous new characterizations of the Sobolev spaces $W^{k,p}(\R^N)$ or $BV(\R^N)$  have been obtained \cite{Bojar,Borghol,Davila,Ferreira,Ponce,Ponce2,Spector} and various asymptotic formulas characterizing the Sobolev norms in terms of fractional norms have been derived \cite{Karadzhov,Kolyada,Mazya,Triebel2}. For instance, Maz'ya and Shaposhnikova \cite{Mazya} obtained the counterpart of \eqref{BBMlimimb} in the critical case $r\downarrow0$, that is
\begin{align}
\lim_{r\downarrow0}\, rp\|f\|_{W^{r,p}(\R^N)}^p=2\sigma_N\|f\|_{L^p(\R^N)}^p,
\end{align}
whenever $f\in\bigcup_{0<r<1} W^{r,p}(\R^N)$ and where $\sigma_N$ stands for the superficial measure of the unit sphere $\S^{N-1}$.

%In particular,  A. Ponce \cite{Ponce2} showed that, when $p=1$, \eqref{hypo4} can be replaced by a more general condition of the form
Also, let us mention the work of  Ponce \cite{Ponce2} who was the first to obtain a  characterization  of  the space $BV(\R^N)$ in terms of a class of functions in $L^1(\R^N)$ satisfying
\begin{align}
\int_{\R^N}\int_{\R^N}\rho_\eps(h)~\Omega\left(\frac{|f(x+h)-f(x)|}{|h|}\right)\mathrm{d}x\mathrm{d}h\leq C \quad \text{as } \varepsilon\downarrow0, \label{omo1}
\end{align}
under suitable growth assumptions on $\Omega\in C(\R_+,\R_+)$.

More recently, such type of characterizations were extended  by Borghol \cite{Borghol}, Bojarski, Ihnatsyeva, Kinnunen \cite{Bojar} and Ferreira, Kreisbeck and Ribeiro \cite{Ferreira}, who considered the  cases $1<p<\infty$ in higher order Sobolev spaces. Typically, in \cite{Ferreira} it is shown that the spaces $W^{k,p}(\R^N)$, with $p\in(1,\infty)$ and $k\in\N^\ast$, can be characterized by quantities of the type
\begin{align}
\int_{\R^N}\int_{\R^N} \rho_\eps(h)~\Omega\left(\frac{|\Delta_h^kf(x)|}{|h|^k}\right)\mathrm{d}x\mathrm{d}h, \label{omo2}
\end{align}
where $\Omega:\R_+\to\R_+$ is an increasing, convex function such that
\begin{align}
m_1t^p\leq \Omega(t)\leq m_2t^p,
\end{align}
for all $t\geq0$ and some positive constants $0<m_1<m_2$. \\
% and $\Delta_h^kf$ stands for the usual $k$-th order forward difference of $f$ given by
%\begin{align}
%\Delta_h^kf(x):=\sum_{j=0}^k(-1)^{k-j}\binom{k}{j}f(x+hj),~~x,h\in\R^N. \label{iterateddiff2}
%\end{align}

To our knowledge, very few is known in the case $0<s<M$. Nonetheless, recent works of Lamy and Mironescu \cite{LaMi} suggest a connection between expressions of the type of \eqref{ener-eps} and Besov spaces. In \cite{LaMi}, the authors prove the following
\begin{theorem}[Lamy, Mironescu, \cite{LaMi}]
Let $s>0$, $p,q\in[1,\infty]$ and let $(\rho_\eps)_{\eps>0}\subset L^1(\R^N)$ satisfying \eqref{molli} and such that
\begin{align}
\rho_\eps(h)=\eps^{-N}\rho\left(\frac{h}{\eps}\right) \qquad{\mbox{for some }} ~\rho\in L^1(\R^N). \label{molli3}
\end{align}
Then,
\begin{align}
\|f\|_{B_{p,q}^s(\R^N)}\lesssim\|f\|_{L^p(\R^N)}+\left\|\frac{1}{\eps^s}\|f\ast\rho_\eps-f\|_{L^p(\R^N)}\right\|_{L^q((0,1),\frac{\mathrm{d}\eps}{\eps})}. \label{BLaMi}
\end{align}
\end{theorem}
The converse of this holds under some additional moment condition on $\rho$ (see \cite{LaMi} for further details).
In fact, the case $q=\infty$ is not properly stated nor explicitly proven in \cite{LaMi}. To fill this gap, we shall give some additional details at the end of the paper. A consequence of this, which has not been noticed in \cite{LaMi}, is the following
\begin{prop}\label{nonextension}
Let $s\in(0,1)$, $p\in[1,\infty)$ and $(\rho_\varepsilon)_{\varepsilon>0}\subset L^1(\R^N)$ satisfying \eqref{molli} and \eqref{molli3}.
Then, the following are equivalent:
\begin{enumerate}
\item[(i)] $f\in B_{p,\infty}^s(\mathbb{R}^N)$,
\item[(ii)] $f\in L^p(\mathbb{R}^N)$ satisfies
\begin{align}
\int_{\mathbb{R}^N}\int_{\mathbb{R}^N}\rho_\varepsilon(h)\frac{|f(x+h)-f(x)|^p}{|h|^{sp}}\mathrm{d}x\mathrm{d}h\leq C \qquad{\mbox{as }} ~\eps\downarrow0. \label{kpp}
\end{align}
\end{enumerate}
Moreover,
\begin{align}
\|f\|_{B_{p,\infty}^s(\mathbb{R}^N)}^p\sim \|f\|_{L^p(\R^N)}^p+\sup_{\varepsilon\in(0,1)}\int_{\mathbb{R}^N}\int_{\mathbb{R}^N}\rho_\varepsilon(h)\frac{|f(x+h)-f(x)|^p}{|h|^{sp}}\mathrm{d}x\mathrm{d}h. \label{Kl2}
\end{align}
\end{prop}
It is worth noticing that, by contrast with the representation of $B_{p,\infty}^s(\R^N)$ obtained in \cite{LaMi}, no moment condition on $\rho_\varepsilon$ is needed. Moreover, since $\rho_\eps$ does not need to be radial some directions may be privileged, yet with no impact on the resulting norm. This is in clear contrast with the case $s=1$ (see also \cite[Remark 10]{Chiron} or \cite[Corollary 3, p.232]{Ponce2}).

This sheds new lights on how to describe smoothness
%via expressions of the type \eqref{hypo4}
and could be of potential interest in some problems of the calculus of variations and in the study of some integro-differential equations (see e.g. \cite{NDF,Aubert,BCV,Osher1,Osher2,Trageser}).

Also, in view of Theorem \ref{nonextension}, it is natural to ask for corresponding assertions of \eqref{omo1} and \eqref{omo2} in the framework of the fractional Besov-Nikol'skii spaces $B_{p,\infty}^s(\R^N)$. For example: what can be said about the limiting behavior of \eqref{kpp} when $\eps\downarrow0$ ? Can one describe higher order Besov-Nikol'skii spaces via expressions of the type \eqref{ener-eps} ? It is the main concern of this paper to deal with these issues.

\subsection{Main Motivation}
%This work was initially motivated by an open problem raised in a recent paper of H. Berestycki, J. Coville and H.-H. Vo \cite{BCV} involving quantities of the type of those considered in this paper and which have not been yet systematically studied. \\
This work originates in a problem raised in \cite{BCV}. Consider the heterogeneous Fisher-KPP equation:
\begin{align}
\frac{1}{\varepsilon^m}\big(J_\varepsilon\ast u(x)-u(x)\big)+f(x,u)=0,~\,\,~u=u_\varepsilon,~\,\,~x\in\mathbb{R}^N,~\,\,~\eps>0, \label{KPP}
\end{align}
where $m\in[0,2]$, $u$ is the density of a given population, $J_\varepsilon(z):=\frac{1}{\varepsilon^N}J\big(\frac{z}{\varepsilon}\big)$, with $J\in C\cap L^1(\mathbb{R}^N)$ a symmetric positive dispersal kernel with unit mass and having finite $m$-th order moment, and $f\in C^{1,\alpha}(\R^{N+1})$ is a heterogeneous KPP type non-linearity, that is:
\begin{align}
\left\{
\begin{array}{l}
f(\cdot,0)=0, \\
\text{for all}~~x\in\R^N,~f(x,s)/s~\text{is decreasing with respect to}~s\in(0,\infty), \\
\text{there exists}~S(x)\in C(\R^N)\cap L^\infty(\R^N)~\text{such that}~f(\cdot,S(\cdot))\leq 0.
\end{array}
\right. \nonumber
\end{align}
For the sake of simplicity, we restrict our attention to non-linearities of the form
\[f(x,s)=s(a(x)-s), \quad \text{ with } \quad \limsup_{|x|\to\infty}\,a(x)<0.\]
Roughly speaking, $f$ models the growth rate of the population and $J$ the probability to jump from one location to another. The parameter $\eps$ is a measure of the spread of dispersal of the species. The scaling term $\frac{1}{\varepsilon^m}$ can be interpreted as the rate of dispersal of the species. It arises when considering a cost function (see \cite[Section 2]{BCV} for a more detailed explanation on the matter).
Consider for instance a tree reproducing and dispersing seeds.
Then, $\varepsilon\ll1$ represents a strategy where the dispersal rate is large but the seeds are spread over smaller distances, and $\varepsilon\gg1$ represents the opposite strategy (i.e. smaller dispersal rate but the seeds are spread over larger distances). As for
the parameter $m$, it measures the influence of the cost function on the different strategies.

Existence of positive solutions to \eqref{KPP} is naturally expected to provide a persistence criteria for the population under consideration. Nonetheless, if the asymptotic of a solution of \eqref{KPP} are quite well understood when $\varepsilon\to\infty$ (see \cite{BCV}), it is not the case when $\varepsilon\downarrow0$ and $0<m<2$. Berestycki et al. \cite{BCV} were able to prove the
\begin{theorem}[Berestycki, Coville, Vo, \cite{BCV}] \label{BCV1}
Assume $J$ is compactly supported with $J(0)>0$, $m\in(0,2)$, $\max\{a,0\}\not\equiv0$ and $a\in C^2(\R^N)$.

Then, when $\varepsilon\downarrow0$, the solution $u_\varepsilon$ of \eqref{KPP} converges almost everywhere to some non-negative bounded function $v$ satisfying
\begin{align}
v(x)\big(a(x)-v(x)\big)=0 \qquad{\mbox{in }} ~\mathbb{R}^N. \label{equationv}
\end{align}
\end{theorem}
Unfortunately, equation \eqref{equationv} admits infinitely many solutions, so it may happen that $v\equiv0$ (extinction) or that $v=a_+\mathds{1}_{K}$ for some compact $K\subset\mathrm{supp}(a_+)$ (persistence in a given area of the ecological niche). Whence, one cannot directly infer a persistence strategy for that case.

However, it is known that solutions to \eqref{KPP}, when they exist, satisfy
\begin{align}
\int_{\mathbb{R}^N}\int_{\mathbb{R}^N}\rho_\varepsilon(x-y)\frac{|u_\varepsilon(x)-u_\varepsilon(y)|^2}{|x-y|^m}\mathrm{d}x\mathrm{d}y\leq C \qquad{\mbox{for all }} ~\varepsilon>0, \label{BBM-KPP}
\end{align}
with $\rho_\varepsilon(z)=\eps^{-m}|z|^mJ_\eps(z)$ a smooth mollifier satisfying \eqref{molli} (see \cite[Lemma 5.1(ii)]{BCV} for a proof).

To quote Berestycki et al.: "\emph{If for the case $m=2$ we could rely on elliptic regularity and the new description of Sobolev Spaces developed in Bourgain et al. \cite{BBM}, Brezis \cite{Brezis}, Ponce \cite{Ponce,Ponce2} to get some compactness, this characterization does not allow us to treat the case $m<2$. We believe that a new characterization of fractional Sobolev spaces in the spirit of the work of Bourgain, Brezis and Mironescu \cite{BBM,Brezis} will be helful to resolve this issue.}"

This motivates the study of general classes of functions of the type of \eqref{ener-eps}, in particular the forthcoming Theorem \ref{THEO} and Theorem \ref{noncompactness}.

\subsection{Comments}
If \eqref{kpp} is very similar to \eqref{hypo4}, the underlying spaces, $W^{1,p}(\R^N)$ and $B_{p,\infty}^s(\R^N)$, are very different in nature and one has to cope with some technicalities. Among others, it is not clear anymore whether the limit of \eqref{kpp} as $\eps\downarrow0$ exists nor, even if it does, whether it provides an equivalent semi-norm. In the integer order case, things are not too controversial in the sense that
\begin{align}
\|\nabla f\|_{L^p(\R^N)}\sim\limsup_{|h|\to0}\frac{\|\Delta_h^1f\|_{L^p(\R^N)}}{|h|}=\sup_{h\ne0}\frac{\|\Delta_h^1f\|_{L^p(\R^N)}}{|h|}. \label{entier}
\end{align}
(see e.g. \cite{Triebel2} or Lemma \ref{limEsup}), while the counterpart of \eqref{entier} in the fractional case $s\in(0,1)$ is not true in general. Indeed, every nontrivial function $f\in C_c^\infty(\R^N)$ satisfies
\begin{align}
\lim_{|h|\to0}\frac{\|\Delta_h^1f\|_{L^p(\R^N)}}{|h|^s}=0<\sup_{h\ne0}\frac{\|\Delta_h^1f\|_{L^p(\R^N)}}{|h|^s}=[f]_{B_{p,\infty}^s(\R^N)}, \label{NBkk}
\end{align}
whenever $s\in(0,1)$, $p\in[1,\infty]$. Finiteness of either or both the two first terms in the left-hand side of \eqref{NBkk} equally describes $B_{p,\infty}^s(\R^N)$ in the sense that they define the same set of functions. But the respective (semi-)norms induced by these quantities are not equivalent (see Section \ref{sectionN}). For these reasons, at some places, it will be more convenient to state our results in terms of suprema as in \eqref{Kl2} instead of limits.

On the other hand, smooth functions are not dense in $B_{p,\infty}^s(\R^N)$, so that the arguments used in the integer case do not simply adapt. We show how to do this in a way that allows, not only to give a meaning, but also to handle the tricky case $p=\infty$ in both the integer and the fractional case, using only elementary arguments. Also, in the particular case where $\rho$ is radially symmetric, we improve \eqref{Kl2} to a semi-norm equivalence at all orders $s>0$. More general quantities are also investigated as well as compactness in the case of a sequence $(f_\eps)_{\eps>0}\subset L^p(\R^N)$.

At the end, this yields a common nonlocal description for the Besov-Nikol'skii spaces $B_{p,\infty}^s(\R^N)$, the H\"older-Zygmund spaces $\mathcal{C}^s(\R^N)$, the $BV(\R^N)$ space, the Sobolev spaces $W^{k,p}(\R^N)$ and the Lipschitz space $C^{0,1}(\R^N)$. As a by-product, we obtain new characterizations for these spaces and a new limiting embedding between Lipschitz and Besov spaces which extends the previous known results (see Theorem \ref{imbLip}). \\

%This work was initially motivated by an open problem raised in a recent paper of H. Berestycki, J. Coville and H.-H. Vo \cite{BCV} involving quantities of the type of those considered in this paper and which have not been yet systematically studied. \\
\section{Main results}

\subsection{A new characterization of Besov-Nikol'skii spaces}
To state our results, we shall introduce some notations and terminology.
\begin{defn}\label{rsubad}
A function $\omega:\R_+\to\R_+$ is said to be \emph{roughly subadditive} if there exists a constant $A>0$ such that,
\[\omega(t_1+t_2)\leq A\left\{\omega(t_1)+\omega(t_2)\right\},\]
for every $t_1,t_2\in\R_+$. When $A=1$ we say that $\omega$ is \emph{subadditive}.
\end{defn}
To shorten our statements, it will be more convenient to call $C_{\mathrm{inc}}$ the set of all continuous, increasing functions $\omega:[0,\infty)\to[0,\infty)$ satisfying $\omega(0)=0$ and $\lim_{t\to\infty}~\omega(t)=\infty$. Also, we set
\begin{align}
C_{\mathrm{inc}}^+:=\{~\omega\in C_{\mathrm{inc}}~~\text{such that}~~\omega~~\text{is roughly subadditive }\}.
\end{align}
\begin{rem}\label{cinc}
Observe that if $\omega_1,\omega_2\in C_{\mathrm{inc}}^+$, then $\omega_1\circ\omega_2\in C_{\mathrm{inc}}^+$.
\end{rem}
Typical examples of functions in $C_{\mathrm{inc}}^+$ are:
\begin{enumerate}
\begin{minipage}[t]{0.5\linewidth}
\item[(i)] $\omega_1(t)=t^\alpha$ with $\alpha>0$,
\item[(ii)] $\omega_2(t)=\ln(1+t)$,
\end{minipage}
\begin{minipage}[t]{0.5\linewidth}
\item[(iii)] $\omega_3(t)=t~\mathrm{tanh}(t)$,
\item[(iv)] $\omega_4(t)=\mathrm{arsinh}(t)$, ...
\end{minipage}
\end{enumerate}
More generally, if $\omega\in C_{\mathrm{inc}}$ is concave, then $\omega\in C_{\mathrm{inc}}^+$ (see e.g. \cite[Theorem 5]{Bruckner}). As indicated by the example of $t^\alpha$ with $\alpha>1$, $C_{\mathrm{inc}}^+$ contains also some convex functions as long as they do not increase too fast. Indeed, a direct computation shows that if $\omega:[0,\infty)\to[0,\infty)$ is a continuous, convex function with $\omega(0)=0$ and if $\omega(2t)\leq \kappa~\omega(t)$, for all $t\geq0$ and some constant $\kappa>0$ (independent of $t$), then $\omega\in C_{\mathrm{inc}}^+$.
%In addition, given $p\in[1,\infty)$, we call $C_p$ the set of all functions $\omega:\R_+\to\R_+$ such that $m_1t^p\leq\omega(t)\leq m_2t^p$~ for some $0<m_1\leq m_2$ and all $t\geq0$.

Our first result reads as follows
\begin{theorem}\label{THEO}
Let $M\in\N^\ast$, $s\in(0,M)$ and $p\in[1,\infty]$. Let $\omega\in C_{\mathrm{inc}}^+$ and $(\rho_\eps)_{\eps>0}\subset L^1(\R^N)$ be a sequence of radial functions satisfying \eqref{molli} and \eqref{molli3}.
Then, the following are equivalent:
\begin{enumerate}
\item[(i)] $f\in B_{p,\infty}^s(\R^N)$,
\item[(ii)] $f\in L^p(\R^N)$ is such that
\begin{align}
\int_{\R^N}\rho_\eps(h)~\omega\left(\frac{\|\Delta_h^Mf\|_{L^p(\R^N)}}{|h|^s}\right)\mathrm{d}h\leq C \qquad{\mbox{as }} ~\eps\downarrow0. \label{LAMBDA}
\end{align}
\end{enumerate}
Moreover,
\[\omega\left([f]_{B_{p,\infty}^s(\R^N)}\right)\sim\,\sup_{\eps>0}\,\int_{\R^N}\rho_\eps(h)~\omega\left(\frac{\|\Delta_h^Mf\|_{L^p(\R^N)}}{|h|^s}\right)\mathrm{d}h.\]
\end{theorem}
%\begin{rem}
%In particular, $\Lambda_{p,\omega}^{s,M,\rho}(\R^N)$ coincides with
%\begin{align}
%B_{p,\infty}^s(\R^N)&~~\text{if}~~s<M~~\text{and}~~1\leq p<\infty, \label{casNikolskii} \\
%\mathcal{C}^s(\R^N)&~~\text{if}~~s<M~~\text{and}~~p=\infty, \label{casHolder} \\
%BV(\R^N)&~~\text{if}~~s=M=1~~\text{and}~~p=1, \label{casBV} \\
%W^{1,p}(\R^N)&~~\text{if}~~s=M=1~~\text{and}~~1<p<\infty, \label{casSobolev} \\
%C^{0,1}(\R^N)&~~\text{if}~~s=M=1~~\text{and}~~p=\infty. \label{casLipschitz}
%\end{align}
%When $1\leq p<\infty$ and $\omega=\left|\cdot\right|^p$ the corresponding assertions for $s\geq M\geq2$ are essentially covered by \cite{Bojar,Borghol,Ferreira} (we do not know if this is still true for general $\omega$). Assertions \eqref{casBV}-\eqref{casSobolev} correspond to a slight generalization of the original result of Bourgain, Brezis, Mironescu (and their followers, e.g. \cite{Bojar,Borghol,Davila,Ferreira,Ponce2,Spector}) in the particular case where $\rho_\eps$ satisfies \eqref{molli3}. Assertion \eqref{casLipschitz}, however, is new.
%\end{rem}
\begin{rem}\label{rkhypo}
It is noteworthy that the assumptions of Theorem \ref{THEO} are somehow self-improving. For example, if $\omega\in C_{\mathrm{inc}}$ is such that
\begin{align}
\alpha_1~\underline{\omega}\leq\omega\leq \alpha_2~ \overline{\omega} \qquad{\mbox{a.e. in }} ~\R^N,
\end{align}
for some $\underline{\omega},~\overline{\omega}\in C_{\mathrm{inc}}^+$ and $\alpha_1,\alpha_2>0$, then $\omega$ still characterizes $B_{p,\infty}^s(\R^N)$. Note also that the Jensen inequality allows to extend this result to convex $\omega\in C_{\mathrm{inc}}$.

Moreover, the conclusion of Theorem \ref{THEO} still holds under the slightly weaker assumption that $(\rho_\eps)_{\eps>0}\subset L^1(\R^N)$ satisfies \eqref{molli} and \eqref{molli3} with $\rho\in L^1(\R^N)$ such that there exists a number $\delta>0$ and a nonnegative radial function $\varphi$ with $\rho\geq\varphi$ a.e. in $B_{\delta}$ and $\int_{B_\delta}\varphi>0$.

Also, when $1\leq p<\infty$, the fact that $\omega\in C_{\mathrm{inc}}^+$ allows one to replace \eqref{LAMBDA} by
\begin{align}
\int_{\R^N}\rho_\eps(h)~\omega\left(\int_{\R^N}\Omega\left(\frac{|\Delta_h^Mf(x)|}{|h|^s}\right)\mathrm{d}x\right)\mathrm{d}h, \label{LAMBDA2}
\end{align}
for any continuously increasing $\Omega:\R_+\to\R_+$ with $\Omega(0)=0$ and
\begin{align}
m_1t^p\leq \Omega(t)\leq m_2t^p, \label{OOmg1}
\end{align}
for all $t\geq0$ and some $0<m_1\leq m_2$.
%As already noted in \cite{Ferreira}, expressions of the type \eqref{LAMBDA2} have recently appeared in applications involving image denoising models with nonlocal regularization terms, see e.g. \cite{Osher1,Osher2}.
\end{rem}
By the same token, we obtain the following counterpart for the Lipschitz space.
\begin{theorem}\label{THEO2}
Let $\omega\in C_{\mathrm{inc}}^+$ and $(\rho_\eps)_{\eps>0}\subset L^1(\R^N)$ be a sequence of radial functions satisfying \eqref{molli} and \eqref{molli3}. Then, the following are equivalent:
\begin{enumerate}
\item[(i)] $f\in C^{0,1}(\R^N)$,
\item[(ii)] $f\in L^\infty(\R^N)$ is such that
\begin{align}
\int_{\R^N}\rho_\eps(h)~\omega\left(\frac{\|f(\cdot+h)-f\|_{L^\infty(\R^N)}}{|h|}\right)\mathrm{d}h\leq C \qquad{\mbox{as }} ~\eps\downarrow0. \label{LAMBDA3}
\end{align}
\end{enumerate}
Moreover,
\[\omega\left([f]_{C^{0,1}(\R^N)}\right)\sim\limsup_{\eps\downarrow0}~\int_{\R^N}\rho_\eps(h)~\omega\left(\frac{\|f(\cdot+h)-f\|_{L^\infty(\R^N)}}{|h|}\right)\mathrm{d}h.\]
\end{theorem}
In fact, our proof also allows to cover first order Sobolev spaces. For example, in view of \eqref{LAMBDA2}, we have the
\begin{theorem}\label{THEO3}
Let $1\leq p< \infty$, $(\omega,\Omega)\in C_{\mathrm{inc}}^+\times C_{\mathrm{inc}}$ with $\Omega$ satisfying \eqref{OOmg1} and $(\rho_\eps)_{\eps>0}\subset L^1(\R^N)$ be a sequence of radial functions satisfying \eqref{molli} and \eqref{molli3}.

Then, the following are equivalent:
\begin{enumerate}
\item[(i)] $f\in W^{1,p}(\R^N)$ $($resp. $f\in BV(\R^N)$ if $p=1)$,
\item[(ii)] $f\in L^p(\R^N)$ is such that
\begin{align}
\int_{\R^N}\rho_\eps(h)~\omega\left(\int_{\R^N}\Omega\left(\frac{|f(x+h)-f(x)|}{|h|}\right)\mathrm{d}x\right)\mathrm{d}h\leq C \qquad{\mbox{as }} ~\eps\downarrow0. \nonumber
\end{align}
\end{enumerate}
Moreover,
\begin{align}
\omega\left(\|\nabla f\|_{L^p(\R^N)}^p\right)\sim\limsup_{\eps\downarrow0}\int_{\R^N}\rho_\eps(h)~\omega\left(\int_{\R^N}\Omega\left(\frac{|f(x+h)-f(x)|}{|h|}\right)\mathrm{d}x\right)\mathrm{d}h. \label{sobolev}
\end{align}
\end{theorem}
Note that the limit superior in the right-hand side of \eqref{sobolev} may not necessarily coincide with the limit inferior, depending on the choices of $\omega$ and $\Omega$.
\begin{rem}
If $\omega(t)=t$ and $\Omega$ is convex, then the corresponding assertion still holds in higher order Sobolev spaces, see \cite{Ferreira} for a proof.
\end{rem}
Here are some straightforward consequences of Theorem \ref{THEO}.
\begin{ex}
Let $M\in\N^\ast$, $s\in(0,M)$ and $J\in L^1(\R^N)$ be a radial function such that
\[\mathcal{J}:=\int_{\R^N}J(z)|z|^{sq}\mathrm{d}h<\infty \qquad{\mbox{for some }} ~1\leq q<\infty.\]
Then, choosing
$$\rho_\eps(h)=\frac{1}{\mathcal{J}}\frac{|h|^{sq}}{\eps^{sq}}J_\eps(h),$$
and $\o(t)=t^q$ we obtain
\begin{align}
[f]_{B_{p,\infty}^s(\R^N)}\sim\,\sup_{\eps>0}~\left(\frac{1}{\eps^{sq}}\int_{\R^N}J_\eps(h)\|\Delta_h^Mf\|_{L^p(\R^N)}^q\mathrm{d}h\right)^{1/q}. \label{KPPkernel}
\end{align}
\end{ex}
\begin{rem}
Notice that the quantity \eqref{BBM-KPP} appearing in the study of the nonlocal Fisher-KPP equation \eqref{KPP} can be seen as a particular case of \eqref{KPPkernel}.
\end{rem}
Other choices of $\rho_\eps$ highlight interesting links with the more classical Besov spaces $B_{p,q}^s(\R^N)$ with $1\leq q<\infty$ (see Definition \ref{DEFBESOV} on Section \ref{ntn} for the definition of these spaces).
\begin{ex}\label{sobolN}
Given $1\leq q<\infty$, the choice $\o(t)=t^q$ and
\begin{align}
\rho_\eps(h)=\frac{1}{C|h|^N}\,\mathds{1}_{(\eps,2\eps)}(|h|), \label{choice2}
\end{align}
where $C=\sigma_N\ln(2)$, yields
\begin{align}
[f]_{B_{p,\infty}^s(\R^N)}\sim\, \sup_{\eps>0}\left(\int_{\eps<|h|<2\eps}\frac{\|\Delta_h^Mf\|_{L^p(\R^N)}^q}{|h|^{N+sq}}\mathrm{d}h\right)^{1/q}. \nonumber \nonumber
\end{align}
\end{ex}
\begin{ex}\label{imbnikol}
Given $1\leq q<\infty$, the choice $\o(t)=t^q$ and
$$\rho_\eps(h)=\frac{1}{\sigma_N\eps^{(s-r)q}}\frac{(s-r)q}{|h|^{N-(s-r) q}}\,\mathds{1}_{(0,\eps)}(|h|),$$
for some $r\in(0,s)$, gives
\begin{align}
q^{-1/q}[f]_{B_{p,\infty}^s(\R^N)}\sim\sup_{\eps>0}~\frac{(s-r)^{1/q}}{\eps^{s-r}}\left(\int_{|h|<\eps}\frac{\|\Delta_h^Mf\|_{L^p(\R^N)}^q}{|h|^{N+rq}}\mathrm{d}h\right)^{1/q}.
\end{align}
\end{ex} ~~

\subsection{Limits of Besov norms}

Following the original result of Bourgain, Brezis and Mironescu in \cite{BBM}; Karadzhov, Milman, Xiao \cite{Karadzhov} and Triebel \cite{Triebel2} proved the following limiting embedding
\begin{align}
q^{-1/q}\|\nabla f\|_{L^p(\R^N)}\sim \lim_{r\uparrow1}~(1-r)^{1/q}[f]_{B_{p,q}^r(\R^N)}, \quad \text{ for }~~1<p,q<\infty. \label{TriebBesov}
\end{align}
See e.g. \cite{Triebel2} where higher order derivatives are also studied.

The counterpart of Example \ref{imbnikol} for the Lipschitz space leads one to ask wether \eqref{TriebBesov} still holds in the critical case $p=\infty$ (recall $W^{1,\infty}(\R^N)$ is the same as $C^{0,1}(\R^N)$). However, because of the restriction to \eqref{molli3} in Theorem \ref{THEO2} one cannot directly infer that this is the case. In addition, spaces of the type $W^{1,\infty}(\R^N)$ or $B_{\infty,q}^s(\R^N)$ do not admit nice spaces such as $C_c^\infty(\R^N)$ as dense subset (they are not even separable) and they inherit from the "bad" properties of $L^\infty(\R^N)$. This makes the validity of \eqref{TriebBesov} in the case $p=\infty$ rather unclear.

We prove that a weaker version of \eqref{TriebBesov} still holds when $p=\infty$.
\begin{theorem}\label{imbLip}
Let $q\in[1,\infty)$ and assume $f\in L^\infty(\R^N)$ is such that
\begin{align}
\limsup_{r\uparrow1}~(1-r)^{1/q}\|f\|_{B_{\infty,q}^r(\R^N)}<\infty. \label{limSUP0}
\end{align}
Then, $f\in C^{0,1}(\R^N)$. Moreover,
\begin{align}
q^{-1/q}[f]_{C^{0,1}(\R^N)}\sim\limsup_{r\uparrow1}~(1-r)^{1/q}\|f\|_{B_{\infty,q}^r(\R^N)}. \label{ekivBL}
\end{align}
\end{theorem}
\begin{rem}
Due to the lack of continuity of the translations in $L^\infty(\R^N)$ it is not clear wether the $\limsup$ in \eqref{limSUP0} (resp. in \eqref{ekivBL}) can be replaced by a $\liminf$.
\end{rem}
The proof can be carried out using subadditivity and monotonicity arguments via an improvement of the Chebychev inequality due to Bourgain, Brezis and Mironescu \cite{BBM} together with Theorem \ref{THEO}.

However, in the fractional case, one lose the aforementioned monotonicity and the arguments fail. In view of Example \ref{imbnikol} and $\mathcal{C}^s(\R^N)=B_{\infty,\infty}^s(\R^N)$ it is natural to ask wether or not the counterpart holds for $B_{p,\infty}^s(\R^N)$.

Using subatomic decompositions we were able to show that this is not the case.
\begin{theorem}\label{nonlimit}
Let $s>0$, $p\in[1,\infty)$ and $q\in[1,\infty)$. Then, there exists a function $f$ belonging to $L^p(\R^N)$, satisfying
\begin{align}
\sup_{0<r<s}\,(s-r)^{1/q}\|f\|_{\mathbf{B}_{p,q}^{r}(\R^N)}<\infty,
\end{align}
but $f\notin B_{p,\infty}^s(\R^N)$.
\end{theorem}
Here, "$\|f\|_{\mathbf{B}_{p,q}^{r}(\R^N)}$" stands for the $B_{p,q}^{r}(\R^N)$-norm of $f$ in the sense of subatomic decomposition theory (see Definition \ref{DEFquarkonial} below).

In particular, this suggests that the restriction to \eqref{molli3} in Theorems \ref{nonextension} and \ref{THEO} (and actually also in \eqref{BLaMi} when $q=\infty$) is not far from being optimal.

\subsection{A non-compactness result}

In the integer case $s=1$, it is known that any bounded sequence $(f_\eps)_{\eps>0}\subset L^p(\R^N)$ satisfying
\begin{align}
\int_{\R^N}\int_{\R^N}\rho_\eps(h)\frac{|f_\eps(x+h)-f_\eps(x)|^p}{|h|^{p}}\mathrm{d}x\mathrm{d}h\leq C \qquad{\mbox{as }} ~\eps\downarrow0,
\end{align}
must be relatively compact in $L_{loc}^p(\R^N)$ provided $(\rho_\eps)_{\eps>0}$ is a suitable sequence of mollifiers (e.g. nonincreasing if $N=1$ \cite{BBM} or radially symmetric if $N\geq2$ \cite{Ponce}).

Per contra, we show that this phenomenon does not extend to $s\in\R_+\setminus\N$, at least if $\rho_\eps$ exhibits a reasonable decay at infinity.
\begin{theorem}\label{noncompactness}
Let $M\in\N^\ast$, $s\in(0,M)$ and $p\in[1,\infty)$. Let $(\rho_\eps)_{\eps>0}$ be a sequence of mollifiers of the form \eqref{molli3} with $\rho\in L^1(\R^N)$ satisfying the moment condition
\begin{align}
\int_{\R^N}\rho(z)|z|^{p(M-s)}\mathrm{d}z<\infty.
\end{align}
Then, there exists a bounded sequence $(f_\eps)_{\eps>0}\subset L^p(\R^N)$ satisfying
\begin{align}
\int_{\R^N}\int_{\R^N}\rho_\eps(h)\frac{|\Delta_h^Mf_\eps(x)|^p}{|h|^{sp}}\mathrm{d}x\mathrm{d}h\leq C \qquad{\mbox{as }} ~\eps\downarrow0,
\end{align}
but which is not relatively compact in $L_{loc}^p(\R^N)$.
\end{theorem}

%In fact, even the pointwise a.e. convergence of $f_\eps$ to some function $f\in B_{p,\infty}^s(\R^N)$ is not sufficient to conclude that $(f_\eps)_{\eps>0}$ is relatively compact in $L_{loc}^p(\R^N)$.
\begin{rem}
In some particular cases it is possible to get rid of assumption \eqref{molli3}. For instance, if the $\rho_\eps$ are non-increasing and supported in some ball of the form $B_{r\eps}$ for all $\eps>0$ and some $r>0$, then the result still holds. Notice also that the conclusion of Theorem \ref{noncompactness} still holds for slightly more general functionals in the spirit of \eqref{LAMBDA} with, say, $\omega=\left|\cdot\right|^{q/p}$, $\Omega=\left|\cdot\right|^{p}$, for any $q\geq1$.
\end{rem}
In the same vein, we obtain the following
\begin{theorem}\label{noncpctB}
Let $s>0$, $p\in[1,\infty)$ and $q\in[1,\infty)$.
Then, there exists a bounded sequence $(f_\eps)_{\eps>0}\subset L^p(\R^N)$ satisfying
\begin{align}
\limsup_{\eps\downarrow0}\,\eps\|f_\eps\|_{B_{p,q}^{s-\eps}(\R^N)^\ast}^q<\infty, \label{Ccpct}
\end{align}
but which is not relatively compact in $L_{loc}^p(\R^N)$.
\end{theorem}
The subscript ''$\ast$" in \eqref{Ccpct} means that the $B_{p,q}^{s-\eps}$-norm of $f_\eps$ is calculated using $(\lfloor s\rfloor+1)$-th order finite differences (according to Definition \ref{DEFBESOV}).
This is no longer true if, instead, we use smaller order differences. For example, if $(f_\eps)_{\eps>0}$ is bounded in $L^p(\R^N)$, then
\begin{align}
\limsup_{\eps\downarrow0}\,\eps\|f_\eps\|_{W^{1-\eps,p}(\R^N)}^p<\infty, \label{classic}
\end{align}
implies that $(f_\eps)_{\eps>0}$ is relatively compact in $L_{loc}^p(\R^N)$, while
\begin{align}
\limsup_{\eps\downarrow0}\,\eps\|f_\eps\|_{B_{p,p}^{1-\eps}(\R^N)^\ast}^p<\infty,
\end{align}
does not. Evidently, this restriction is immaterial if $0<s\notin\N$.

\subsection{An approximation criteria}
It is well-known that neither $C_c^\infty(\R^N)$ nor $\mathscr{S}(\R^N)$ are dense in $B_{p,\infty}^s(\R^N)$. If the question of how to approximate a given $f\in B_{p,q}^s(\R^N)$ in a "suitable manner" has already been well-studied (see e.g. \cite{Kowalski,Nikolskii,Stasyuk}), to the author's knowledge it seems, however, that no criterion to recognize a function $f\in B_{p,\infty}^s(\R^N)$ which can be approximated by smooth functions in its natural (strong) topology is available in the literature.

An interesting consequence of (the proof of) Theorem \ref{THEO} is that it gives such a criterion.
\begin{corollary}\label{approx}
Let $M\in\N^\ast$, $s\in(0,M)$, $p\in[1,\infty)$.
Let $(\rho_\varepsilon)_{\varepsilon>0}\subset L^1(\R^N)$ be a sequence of radial functions satisfying \eqref{molli} and \eqref{molli3}, and let $\omega\in C_{\mathrm{inc}}^+$.
Then, the following are equivalent:
\begin{enumerate}
\item[(i)] $f\in L^p(\R^N)$ is such that
\[\lim_{\varepsilon\downarrow0}\int_{\R^N}\rho_\eps(h)~\omega\left(\frac{\|\Delta_h^Mf\|_{L^p(\R^N)}}{|h|^s}\right)\mathrm{d}h=0,\]
\item[(ii)] $f\in B_{p,\infty}^s(\R^N)$ and there exists $(f_n)_{n\geq0}\subset C_c^\infty(\R^N)$ such that
\[\|f-f_n\|_{B_{p,\infty}^s(\R^N)}\to0 \qquad \text{ as }~~n\to\infty.\]
\end{enumerate}
\end{corollary}
A noteworthy consequence of Corollary \ref{approx} is the following
\begin{ex}
Let $s\in(0,1)$ and $p\in[1,\infty)$. Then, with the choice \eqref{choice2} and $\omega(t)=t^p$ we find that condition (ii) above is equivalent to
\begin{align}
\lim_{\eps\downarrow0}\int_{\R^N}\int_{\eps<|x-y|<2\eps}\frac{|f(x)-f(y)|^p}{~~|x-y|^{N+sp}}\mathrm{d}x\mathrm{d}y=0,
\end{align}
or, more generally, to
\begin{align}
\lim_{\eps\downarrow0}\int_{\eps<|h|<2\eps}\frac{\|\Delta_h^Mf\|_{L^p(\R^N)}^q}{|h|^{N+sq}}\mathrm{d}h=0,
\end{align}
in the higher order case.
\end{ex}

%\subsection{Organization of the paper}

In Sections \ref{ntn} and \ref{sectionsubatom} we detail all our notations and useful definitions. In Section \ref{sectionN}, we show some preliminary estimates which aims to simultaneously open the way to Corollary \ref{approx} and to explain why it is more convenient to represent $B_{p,\infty}^s(\R^N)$ in terms of the supremum of \eqref{LAMBDA} rather than in terms of its limits. Section \ref{sectionNIKOLSKII} is devoted to the proof of Theorem \ref{THEO} and Section \ref{sectionSOBOLEV} to that of Theorems \ref{THEO2}, \ref{THEO3}, and \ref{imbLip}. In Section \ref{sectionLIMITING} we establish Theorem \ref{nonlimit}. In Section \ref{sectionCOMPACT}, we prove Theorems \ref{noncompactness} and \ref{noncpctB}. Finally, in the Appendix, we discuss Proposition \ref{nonextension}.

\section{Notations and definitions}\label{ntn}

Throughout the paper we will make use of the following notations. \\

\begin{tabular}{rl}
$\S^{N-1}$ : & is the unit sphere of $\R^N$; \\
$\mathcal{H}^{N-1}$ : & is the $(N-1)$-dimensional Hausdorff measure; \\
$|K|$ : & is the Lebesgue measure of the set $K$ (also denoted $\lambda_n(K)$); \\
$\mathds{1}_K$ : & is the characteristic function of the set $K$; \\
$B_R$ : & is the ball of radius $R>0$ centered at the origin; \\
$B_R(x)$ : & is the ball of radius $R>0$ centered at $x\in\R^N$; \\
$\tau_h$ : & is the translation operator $\tau_hf(x)=f(x+h)$, $x,h\in\mathbb{R}^N$; \\
$f\ast g$ : & is the convolution of $f$ and $g$; \\
$\lesssim$ : & is the "approximatively-less-than" symbol: $a\lesssim b\Leftrightarrow a\leq Cb$; \\
$\sim$ : & is the equivalence symbol: $a\sim b\Leftrightarrow a\lesssim b~~\text{and}~~b\lesssim a$; \\
$\fint_A$ : & is the integral mean symbol: $\fint_Af=\frac{1}{|A|}\int_Af$.
\end{tabular} \\

We denote by $L^p(\R^N)$ the \emph{Lebesgue space} of (equivalence classes of) functions for which the $p$-th power of the absolute value is Lebesgue integrable (resp. essentially bounded functions when $p=\infty$); by $C_c^\infty(\R^N)$ the space of smooth compactly supported functions; by $\mathscr{S}(\R^N)$ the \emph{Schwartz space} of rapidly decaying functions; and, by $\mathscr{S}'(\R^N)$, its dual, the space of tempered distributions. The \emph{Lipschitz space} $C^{0,1}(\R^N)$ is the space of functions $f\in L^\infty(\R^N)$ for which the semi-norm
\begin{align}
[f]_{C^{0,1}(\R^N)}:=\sup_{h\ne0}\frac{\|\tau_hf-f\|_{L^\infty(\R^N)}}{|h|}, \label{Lipcte}
\end{align}
is finite. The space $C^{0,1}(\R^N)$ is a Banach space for the norm
\[\left\|f\right\|_{C^{0,1}(\R^N)}:=\left\|f\right\|_{L^\infty(\R^N)}+\left[f\right]_{C^{0,1}(\R^N)}.\]
The number \eqref{Lipcte} is called the Lipschitz constant of $f$. For the sake of clarity, we recall some further definitions.
\begin{defn} Let $p\in[1,\infty)$ and $k\in \N^\ast$. The $k$-\emph{th order Sobolev space} $W^{k,p}(\R^N)$ is defined as the closure of $C_c^\infty(\R^N)$ under the norm
\[\|f\|_{W^{k,p}(\R^N)}:=\|f\|_{L^p(\R^N)}+\bigg(\sum_{1\leq|\alpha|\leq k}\|D^\alpha f\|_{L^p(\R^N)}^p\bigg)^{1/p}.\]
\end{defn}
\begin{defn}
The space of \emph{functions of bounded variation}, denoted by $BV(\R^N)$,  is the space of all $f\in L^1(\R^N)$ such that
\[[f]_{BV(\R^N)}:=\sup\left\{\int_{\R^N}f(x)\,\mathrm{div}\,\phi(x)\,\mathrm{d}x : \phi\in C_c^1(\R^N),\,\|\phi\|_{L^\infty(\R^N)}\leq1\right\}<\infty,\]
naturally endowed with the norm
\[\|f\|_{BV(\R^N)}:=\|f\|_{L^1(\R^N)}+[f]_{BV(\R^N)}.\]
\end{defn}
\begin{defn}\label{DEFBESOV}
Let $M\in\N^\ast$, $s\in(0,M)$ and $p,q\in[1,\infty]$. The \emph{Besov space} $B_{p,q}^s(\mathbb{R}^N)$ consists in all functions $f\in L^p(\mathbb{R}^N)$ such that
\begin{align}
[f]_{B_{p,q}^s(\mathbb{R}^N)}:=\left(\int_{\R^N}\|\Delta_h^Mf\|_{L^p(\R^N)}^q\frac{\mathrm{d}h}{|h|^{N+sq}}\right)^{\frac{1}{q}}<\infty, \label{snBesov}
\end{align}
which, in the case $q=\infty$, is to be understood as
\[[f]_{B_{p,\infty}^s(\mathbb{R}^N)}:=\sup_{h\in\mathbb{R}^N\setminus\{0\}}\frac{\|\Delta_h^Mf\|_{L^p(\mathbb{R}^N)}}{|h|^s}<\infty,\]
\end{defn}
\noindent where $\Delta_h^Mf$ is given by \eqref{iterateddiff2}. The space $B_{p,q}^s(\mathbb{R}^N)$ is naturally endowed with the norm
\[\|f\|_{B_{p,q}^s(\mathbb{R}^N)}:=\|f\|_{L^p(\mathbb{R}^N)}+[f]_{B_{p,q}^s(\mathbb{R}^N)}.\]
\begin{rem}\label{entierM}
Of course, if one denotes the semi-norm \eqref{snBesov} by $[f]_{B_{p,q}^s(\R^N)}^{(M)}$, then for $M_1,M_2\in\N^\ast$ with $M_1<M_2$ and $s\in(0,M_1)$ it holds
\[[f]_{B_{p,q}^s(\R^N)}^{(M_1)}\sim[f]_{B_{p,q}^s(\R^N)}^{(M_2)},\]
(similarly when $q=\infty$), so that the definition above is indeed consistent. We refer to \cite{Triebel} (e.g. estimate (45) on p.99) or Lemma \ref{2Morder} for further details.
\end{rem}
\begin{rem}
The integral in \eqref{snBesov} can be indifferently replaced by an integral over $\{|h|\leq \delta\}$ for any $\delta>0$, or on the whole $\mathbb{R}^N$ since the singular part in $h$ in the integral arise when $h$ is close to zero, while the integral on $\{|h|>\delta\}$ can always be dominated by the $L^p$-norm of $f$.
\end{rem}

Of special interest are the cases $q=p$, $p=\infty$ and/or $q=\infty$. The \emph{fractional Sobolev spaces} $W^{s,p}(\mathbb{R}^N)$ (sometimes also called \emph{Slobodeckij}, \emph{Gagliardo}, or \emph{Aronszajn spaces}) is defined by $W^{s,p}(\mathbb{R}^N)=B_{p,p}^s(\mathbb{R}^N)$ for $s\notin\N$. In this context, the semi-norm \eqref{snBesov} when $s\in(0,1)$ is often referred to as the Gagliardo semi-norm. \\

When $q=\infty$, the space $B_{p,\infty}^s(\mathbb{R}^N)$ is called the \emph{Nikol'skii space}. This scale gives another interesting way to measure the convergence rate of the translate of a given function to itself. It is well-known that, for any $p,q\in[1,\infty)$ and $s>0$,
\[B_{p,q}^s(\mathbb{R}^N)\hookrightarrow B_{p,\infty}^s(\mathbb{R}^N),\]
where "$\hookrightarrow$" stands for the continuous imbedding symbol. We refer to \cite{Simon,TriebelI} for a proof of this fact. When $p=q=\infty$, then the space $B_{\infty,\infty}^s(\R^N)$ coincides with the H\"older-Zygmund space $\mathcal{C}^s(\R^N)$. \\

Moreover, by contrast with $W^{s,p}(\mathbb{R}^N)$ (see e.g. \cite{FSV} for a simple proof of this fact) or, more generally, with the spaces $B_{p,q}^s(\R^N)$ with $p,q\in(1,\infty)$, neither $C_c^\infty(\mathbb{R}^N)$ nor $\mathscr{S}(\mathbb{R}^N)$ are dense in $B_{p,\infty}^s(\mathbb{R}^N)$, see e.g. \cite[Theorem 2.3.2 (a), p.172]{TriebelI}. The Nikol'skii spaces are Banach spaces but, unlike, say, $W^{s,p}(\mathbb{R}^N)$ with $1<p<\infty$, neither reflexive \cite[Remark 2, p.199]{TriebelI} nor separable \cite[Theorem 2.11.2 (d), p.237]{TriebelI}.

\section{Subatomic decompositions}\label{sectionsubatom}

There exists many ways to decompose a function $f\in B_{p,q}^s(\R^N)$ into "building blocks". The theory of subatomic (or quarkonial) decompositions developed by Triebel in \cite{TriebelFractal,Triebel3} is one of them of particular interest because, unlike related decompositions of atomic or, say, Littlewood-Paley type, it yields a decomposition of any function $f\in B_{p,q}^s(\R^N)$ on a suitable system of functions which is independent of $f$ and the resulting coefficients are linearly dependent on $f$. In such a framework, the search for a function amounts, roughly speaking, to seeking for a discrete sequence of numbers.

For the convenience of the reader we recall some basic definitions.
\begin{defn}\label{quarks}
Let $\nu\geq0$, $m\in\Z^N$ and $\psi\in C_c^\infty(\R^N)$ be a non-negative function with $\mathrm{supp}(\psi)\subset B_{2^r}$ for some $r\geq0$ and
\begin{align}
\sum_{k\in\Z^N}\psi(x-k)=1, \qquad{\mbox{if }} ~x\in\R^N.
\end{align}
Let $Q_{\nu,m}$ be the cube of sides parallel to the coordinate axis with side-length $2^{-\nu}$ and centered at $2^{-\nu}m$.
Let $s\in\R$, $1\leq p\leq\infty$, $\beta\in\N^N$ and
\[ \psi^\beta(x)=x^\beta\psi(x):=x_1^{\beta_1}...x_N^{\beta_N}\psi(x). \]
Then,
\[(\beta\mathrm{qu})_{\nu,m}(x):=2^{-\nu(s-\frac{N}{p})}\psi^\beta(2^\nu x-m), \qquad{\mbox{  }} x\in\R^N, \]
is called an $(s,p)$-$\beta$-quark relative to $Q_{\nu,m}$.
\end{defn}
\begin{defn}\label{DEFquarkonial}
Let $s>0$, $1\leq p,q\leq\infty$ and $(\beta\mathrm{qu})_{\nu,m}$ be $(s,p)$-$\beta$-quarks according to Definition \ref{quarks}. Let $\varrho>r$ where $r$ has the same meaning as in Definition \ref{quarks}. For all $\lambda=\{\lambda_{\nu,m}^\beta\in\mathbb{C}:(\nu,m,\beta)\in\N\times\Z^N\times\N^N\}$ we set
\begin{align}
\|\lambda\|_{\varrho,p,q}:=\sup_{\beta\in\N^N}2^{\varrho|\beta|}\bigg(\sum_{\nu\geq0}\bigg(\sum_{m\in\Z^N}|\lambda_{\nu,m}^\beta|^p\bigg)^{q/p}\bigg)^{1/q},
\end{align}
with obvious modification if $p=\infty$ and/or $q=\infty$.

We call $\mathbf{B}_{p,q}^s(\R^N)$ the collection of all $f\in\mathscr{S}'(\R^N)$ which can be represented as
\begin{align}
f(x)=\sum_{\beta\in\N^N}\sum_{\nu=0}^\infty\sum_{m\in\Z^N}\lambda_{\nu,m}^\beta(\beta\mathrm{qu})_{\nu,m}(x), \label{repre}
\end{align}
endowed with the norm
\begin{align}
\|f\|_{\mathbf{B}_{p,q}^s(\R^N)}:=\inf\,\|\lambda\|_{\varrho,p,q}, \label{quarknorm}
\end{align}
where the infinimum is taken over all admissible representations \eqref{repre}.
\end{defn}
The standard fact of subatomic decompositions states as follows
\begin{theorem}\label{THquark}
Let $s>0$ and $1\leq p,q\leq\infty$. Then, \eqref{quarknorm} does not depend upon the choice of $\varrho$ nor on $\psi$, and $\mathbf{B}_{p,q}^s(\R^N)$ is a Banach space which coincides with the space $B_{p,q}^s(\R^N)$ introduced in Definition \ref{DEFBESOV}. Moreover,
\begin{align}
\|f\|_{\mathbf{B}_{p,q}^s(\R^N)}\sim\|f\|_{B_{p,q}^s(\R^N)}.
\end{align}
\end{theorem}
We refer to \cite{Triebel3} and references therein for a proof of this. In fact, there are \emph{optimal subatomic coefficients}, i.e. coefficients $\lambda_{\nu,m}^\beta(f)$ realizing the infinimum in \eqref{quarknorm} and which can be obtained as a dual pairing of the form $\langle f,\Psi_{\nu,m}^{\beta,\varrho}\rangle_{\mathscr{S}',\mathscr{S}}$ where $(\Psi_{\nu,m}^{\beta,\varrho})\subset\mathscr{S}(\R^N)$ is an appropriate sequence of functions. We refer to \cite{Triebel3} for further details.

\section{The space $N^{s,p}(\mathbb{R}^N)$}\label{sectionN}

The aim of this section is twofold. On the one hand, we point out that, even though the spaces $B_{p,\infty}^s(\R^N)$ can be characterized as limits superior (see Proposition \ref{NversA} below), it does not yield an equivalent norm (as it does for the Sobolev spaces $W^{1,p}(\R^N)$ with $p>1$, see e.g. Lemma \ref{limEsup}). As will become clear in the next section, this is the reason why $B_{p,\infty}^s(\R^N)$ is more conveniently described as the supremum of \eqref{LAMBDA} rather than as its limit superior. On the other hand, we provide some preliminary results towards Corollary \ref{approx}. For simplicity, we consider only first order differences $\Delta_h^1f=\tau_hf-f$ but all the results of this section also hold for higher order differences.
%The aim of this section is twofold: on the one hand, we point out the inconvenience in describing the spaces $B_{p,\infty}^s(\R^N)$ in terms of the limit superior of \eqref{LAMBDA} and, on the other hand, we provide some preliminary explanations towards Corollary \ref{approx}.

For the sake of convenience, we define a "new" function space which, in fact, is merely another way to look at the Nikol'skii space $B_{p,\infty}^s(\mathbb{R}^N)$ as shown hereafter.
\begin{defn}
Let $s\in(0,1)$ and $p\in[1,\infty]$. Then, the space $N^{s,p}(\mathbb{R}^N)$ consists of all functions $f\in L^p(\mathbb{R}^N)$ such that
\[[f]_{N^{s,p}(\mathbb{R}^N)}:=\limsup_{|h|\to0}\frac{\|\tau_hf-f\|_{L^p(\mathbb{R}^N)}}{|h|^s}<\infty.\]
It is endowed with the following norm:
\[\|f\|_{N^{s,p}(\mathbb{R}^N)}:=\|f\|_{L^p(\mathbb{R}^N)}+[f]_{N^{s,p}(\mathbb{R}^N)}.\]
In addition, we also define
\[N_0^{s,p}(\mathbb{R}^N):=\left\{f\in N^{s,p}(\mathbb{R}^N):[f]_{N^{s,p}(\mathbb{R}^N)}=0\right\}.\]
\end{defn}
As expected, we have the
\begin{prop}\label{NversA}
Let $s\in(0,1)$ and $p\in[1,\infty]$. Then,
\[B_{p,\infty}^s(\mathbb{R}^N)=N^{s,p}(\mathbb{R}^N).\]
\end{prop}
\begin{rem}
The equality here holds in the sense of sets: the topology of both are not precisely equivalent as shown below. In fact, "$[\cdot]_{N^{s,p}(\mathbb{R}^N)}$" is a quite crude way to characterize the Nikol'skii space. For these reasons (and in order not to mix with both topologies) we shall write $B_{p,\infty}^s(\mathbb{R}^N)=(B_{p,\infty}^s(\mathbb{R}^N),\left\|\cdot\right\|_{B_{p,\infty}^s(\mathbb{R}^N)})$ and $N^{s,p}(\mathbb{R}^N)=(B_{p,\infty}^s(\mathbb{R}^N),\left\|\cdot\right\|_{N^{s,p}(\mathbb{R}^N)})$.% As our goal is to obtain equivalent (semi-)norms, this justifies why we preferred to represent $B_{p,\infty}^s(\R^N)$ in terms of a supremum over the possibility of doing this in terms of limits as it is done in the classical case.
\end{rem}
\begin{proof}
Let $f\in B_{p,\infty}^s(\mathbb{R}^N)$. Then, for all $\delta>0$, we have
\begin{align}
[f]_{B_{p,\infty}^s(\mathbb{R}^N)}:=\sup_{h\in\mathbb{R}^N\setminus\{0\}}\frac{\|\tau_hf-f\|_{L^p(\mathbb{R}^N)}}{|h|^s}\geq \sup_{0<|h|<\delta}\frac{\|\tau_hf-f\|_{L^p(\mathbb{R}^N)}}{|h|^s}. \nonumber
\end{align}
Letting $\delta\downarrow0$, we get
\begin{align}
[f]_{B_{p,\infty}^s(\mathbb{R}^N)}\geq  \limsup_{|h|\to0}\frac{\|\tau_hf-f\|_{L^p(\mathbb{R}^N)}}{|h|^s}=:[f]_{N^{s,p}(\mathbb{R}^N)}, \label{NversAA}
\end{align}
and so $f\in N^{s,p}(\mathbb{R}^N)$. Conversely, let $f\in N^{s,p}(\mathbb{R}^N)$. Then, for all $\eta>0$ there is a $\delta_0>0$ such that for all $\delta\in(0,\delta_0)$ we have
\[\left|\sup_{0<|h|<\delta}\frac{\|\tau_hf-f\|_{L^p(\mathbb{R}^N)}}{|h|^s}-[f]_{N^{s,p}(\mathbb{R}^N)}\right|<\eta.\]
Now fix such $\eta$ and $\delta$. By the triangle inequality we obtain
\[\sup_{0<|h|<\delta}\frac{\|\tau_hf-f\|_{L^p(\mathbb{R}^N)}}{|h|^s}<\eta+[f]_{N^{s,p}(\mathbb{R}^N)}<\infty.\]
On the other hand,
\[\sup_{\delta\leq|h|}\frac{\|\tau_hf-f\|_{L^p(\mathbb{R}^N)}}{|h|^s}\leq \frac{2}{\delta^s}\|f\|_{L^p(\mathbb{R}^N)}<\infty.\]
Therefore, $f\in B_{p,\infty}^s(\mathbb{R}^N)$.
\end{proof}
\begin{prop}\label{W1versA}
Let $s\in(0,1)$ and $p\in(1,\infty]$. Then,
\[W^{1,p}(\mathbb{R}^N)\subset N_0^{s,p}(\mathbb{R}^N)~~\text{ and  }~~BV(\mathbb{R}^N)\subset N_0^{s,1}(\mathbb{R}^N).\]
\end{prop}
\begin{proof}
First, let $f\in W^{1,p}(\mathbb{R}^N)$ (resp. $f\in BV(\R^N)$ if $p=1$). Then,
\[\frac{\|\tau_hf-f\|_{L^p(\mathbb{R}^N)}}{|h|^{s}}\leq |h|^{1-s}\|\nabla f\|_{L^p(\mathbb{R}^N)}, \quad \forall h\in\mathbb{R}^N.\]
Taking the limit superior as $|h|\to0$ gives $f\in N_0^{s,p}(\mathbb{R}^N)$.
\end{proof}
\begin{prop}\label{cooo}
Let $s\in(0,1)$, $p\in[1,\infty)$ and $\mathring{N}^{s,p}(\mathbb{R}^N)$ denote the closure of $C_c^\infty(\mathbb{R}^N)$ in $N^{s,p}(\mathbb{R}^N)$. Then,
\[\mathring{N}^{s,p}(\mathbb{R}^N)=N_0^{s,p}(\mathbb{R}^N).\]
In particular, $N_0^{s,p}(\mathbb{R}^N)$ is a closed subspace of $N^{s,p}(\mathbb{R}^N)$.
\end{prop}
\begin{proof}
"$\subset$": By definition, $C_c^\infty(\mathbb{R}^N)$ is dense in $\mathring{N}^{s,p}(\mathbb{R}^N)$, whence the inclusion $\mathring{N}^{s,p}(\mathbb{R}^N)\subset N_0^{s,p}(\mathbb{R}^N)$ is straightforward. \\

\noindent "$\supset$": Let $f\in N_0^{s,p}(\R^N)$ and let $(f_n)_{n\geq0}\subset C_c^\infty(\R^N)$ be such that
$$ \|f-f_n\|_{L^p(\R^N)}\to0 \quad \text{ as }~~n\to\infty. $$
Then, clearly,
\begin{align}
\|f-f_n\|_{N^{s,p}(\R^N)}&:=\|f-f_n\|_{L^p(\R^N)}+[f-f_n]_{N^{s,p}(\R^N)} \nonumber \\
&\leq \|f-f_n\|_{L^p(\R^N)}+[f]_{N^{s,p}(\R^N)}+[f_n]_{N^{s,p}(\R^N)} \nonumber \\
&=\|f-f_n\|_{L^p(\R^N)}\longrightarrow0~~~~\text{as}~~~~n\to\infty. \nonumber
\end{align}
Whence, $f\in\mathring{N}^{s,p}(\R^N)$. Moreover, the map
\[\Theta:f\in N^{s,p}(\mathbb{R}^N)\mapsto [f]_{N^{s,p}(\mathbb{R}^N)}\]
is continuous. Therefore $N_0^{s,p}(\mathbb{R}^N)=\Theta^{-1}(\{0\})$ is closed in $N^{s,p}(\mathbb{R}^N)$.
\end{proof}
\begin{prop}\label{completion}
Let $s\in(0,1)$, $p\in[1,\infty)$ and $\mathring{B}_{p,\infty}^s(\mathbb{R}^N)$ (resp. $\mathring{N}^{s,p}(\mathbb{R}^N)$) denote the closure of $C_c^\infty(\mathbb{R}^N)$ in $B_{p,\infty}^s(\mathbb{R}^N)$ (resp. $N^{s,p}(\mathbb{R}^N)$). Then,
\[f\in\mathring{N}^{s,p}(\mathbb{R}^N)~~\text{  if, and only if,  }~~f\in\mathring{B}_{p,\infty}^s(\mathbb{R}^N).\]
\end{prop}
\begin{proof}
Let $f\in \mathring{N}^{s,p}(\mathbb{R}^N)$ and $(f_n)_{n\geq0}\subset C_c^\infty(\mathbb{R}^N)$ be such that
\[f_n\to f \quad \text{ in }~~N^{s,p}(\mathbb{R}^N)~~\text{ as }~~n\to\infty.\]
Thus, for all $\eta>0$ there exists $n_0=n_0(\eta)\geq0$ and $\delta_0=\delta_0(\eta)>0$ such that
\[n\geq n_0,~~\delta\in(0,\delta_0)~~\Rightarrow~~\sup_{|h|<\delta}\frac{\|\Delta_h^1(f-f_n)\|_{L^p(\mathbb{R}^N)}}{|h|^s}<\eta.\]
Now, fix such $\eta$, $\delta$ and $n_0$. On the other hand, for all $\eta>0$ and all $\delta>0$ there is a $n_1=n_1(\eta,\delta)\geq0$ such that
\[n\geq n_1~~\Rightarrow \sup_{|h|\geq\delta}\frac{\|\Delta_h^1(f-f_n)\|_{L^p(\mathbb{R}^N)}}{|h|^s}<\eta.\]
Indeed, this is because
\[\sup_{|h|\geq\delta}\frac{\|\Delta_h^1(f-f_n)\|_{L^p(\mathbb{R}^N)}}{|h|^s}\leq \frac{2}{\delta^s}\|f-f_n\|_{L^p(\mathbb{R}^N)}\to0 \quad \text{ as }~~n\to\infty.\]
Therefore, for all $n\geq\max\{n_0,n_1\}$,
\begin{align}
\sup_{h\ne0}\frac{\|\Delta_h^1(f-f_n)\|_{L^p(\mathbb{R}^N)}}{|h|^s}<\eta.
\end{align}
Summing up, we find that, for all $\eta>0$, there exists $M\geq0$ such that
\[n\geq M~~\Rightarrow~~[f-f_n]_{B_{p,\infty}^s(\mathbb{R}^N)}<\eta.\]
Thus, $f\in\mathring{B}_{p,\infty}^s(\mathbb{R}^N)$. \\

Conversely, let $f\in\mathring{B}_{p,\infty}^s(\mathbb{R}^N)$ and $(f_n)_{n\geq0}\subset C_c^\infty(\mathbb{R}^N)$ be such that $f_n\to f$ in $B_{p,\infty}^s(\mathbb{R}^N)$. Using \eqref{NversAA} we find
\begin{align}
[f]_{N^{s,p}(\mathbb{R}^N)}&\leq [f-f_n]_{N^{s,p}(\mathbb{R}^N)}+[f_n]_{N^{s,p}(\mathbb{R}^N)} \nonumber \\
&=[f-f_n]_{N^{s,p}(\mathbb{R}^N)} \nonumber \\
&\leq [f-f_n]_{B_{p,\infty}^s(\mathbb{R}^N)}\to 0~~\text{as}~~n\to\infty. \nonumber
\end{align}
Thus $f\in \mathring{N}^{s,p}(\mathbb{R}^N)$.
\end{proof}

\section{Characterization of Besov-Nikol'skii spaces}\label{sectionNIKOLSKII}

\subsection{Preliminary}

For the sake of clarity we shall introduce the following short notation
\[\mathscr{D}_\omega(\rho_\eps,f):=\int_{\R^N}\rho_\eps(h)~\omega\left(\frac{\|\Delta_h^Mf\|_{L^p(\R^N)}}{|h|^s}\right)\mathrm{d}h.\]
First, an easy observation.
\begin{prop}\label{cotefacile}
Let $M\in\N^\ast$, $s>0$, $p\in[1,\infty]$ and $(\rho_\varepsilon)_{\varepsilon>0}$ be a sequence of mollifiers. Assume $\omega\in C_{\mathrm{inc}}^+$. Then,
\begin{align}
\limsup_{\eps\downarrow0}\,\mathscr{D}_\omega(\rho_\eps,f)\leq ~\omega\left(\limsup_{|h|\to0}\frac{\|\Delta_h^Mf\|_{L^p(\R^N)}}{|h|^s}\right), \label{Alim}
\end{align}
and
\begin{align}
\sup_{\eps>0}\,\mathscr{D}_\omega(\rho_\eps,f)\leq ~\omega\left(\sup_{h\ne0}\frac{\|\Delta_h^Mf\|_{L^p(\R^N)}}{|h|^s}\right). \label{Asup}
\end{align}
\end{prop}
\begin{proof}
Let $\eta>0$ be any fixed number. Then, we have
\begin{align}
\mathscr{D}_\omega(\rho_\eps,f)= \left(\int_{0\leq|h|\leq\eta}+\int_{|h|>\eta}\right)\rho_\varepsilon(h)~\omega\left(\frac{\|\Delta_h^Mf\|_{L^p(\mathbb{R}^N)}}{|h|^{s}}\right)\mathrm{d}h. \nonumber
\end{align}
On the one hand,
\begin{align}
\int_{0\leq|h|\leq\eta}\rho_\varepsilon(h)~\omega&\left(\frac{\|\Delta_h^Mf\|_{L^p(\mathbb{R}^N)}}{|h|^{s}}\right)\mathrm{d}h \nonumber \\
&\leq \sup_{0\leq|h|\leq\eta}~\omega\left(\frac{\|\Delta_h^Mf\|_{L^p(\mathbb{R}^N)}}{|h|^{s}}\right)\int_{0\leq|h|\leq\eta}\rho_\varepsilon(h)\mathrm{d}h \nonumber \\
&\leq \sup_{0\leq|h|\leq\eta}~\omega\left(\frac{\|\Delta_h^Mf\|_{L^p(\mathbb{R}^N)}}{|h|^{s}}\right). \nonumber
\end{align}
On the other hand, since $\omega$ is non-decreasing
\begin{align}
\int_{|h|>\eta}\rho_\varepsilon(h)~\omega\left(\frac{\|\Delta_h^Mf\|_{L^p(\mathbb{R}^N)}}{|h|^{s}}\right)\mathrm{d}h&\leq \omega\left(\frac{2^{M}\|f\|_{L^p(\mathbb{R}^N)}}{\eta^{s}}\right)\int_{|h|>\eta}\rho_\varepsilon(h)\mathrm{d}h \nonumber \\
&\longrightarrow0~~\text{as}~~\varepsilon\downarrow0. \nonumber
\end{align}
Therefore,
\[\limsup_{\varepsilon\downarrow0}\int_{\mathbb{R}^N}\rho_\varepsilon(h)~\omega\left(\frac{\|\Delta_h^Mf\|_{L^p(\mathbb{R}^N)}}{|h|^{s}}\right)\mathrm{d}h\leq \sup_{0\leq|h|\leq\eta}~\omega\left(\frac{\|\Delta_h^Mf\|_{L^p(\mathbb{R}^N)}}{|h|^{s}}\right).\]
Taking now the limit as $\eta\downarrow0$ and using $\omega\in C_{\mathrm{inc}}^+$ we obtain
\[
\limsup_{\varepsilon\downarrow0}\,\mathscr{D}_\omega(\rho_\eps,f) \leq \omega\left(\limsup_{|h|\to0}\frac{\|\Delta_h^Mf\|_{L^p(\mathbb{R}^N)}}{|h|^{s}}\right).
\]
The remaining inequality follows by a direct application of H\"older's inequality.
\end{proof}
Here is another estimate we shall need.
\begin{lemma}\label{iteratedtrans}
Let $p\in[1,\infty]$, $M\in\N^\ast$, $h_1,h_2\in\R^N$ and $h=h_1+h_2$. Then,
\[\|\Delta_h^{2M}f\|_{L^p(\mathbb{R}^N)}\lesssim \|\Delta_{h_1}^{M}f\|_{L^p(\mathbb{R}^N)}+\|\Delta_{h_2}^{M}f\|_{L^p(\mathbb{R}^N)},\]
for all $f\in L^p(\R^N)$.
\end{lemma}
This is essentially covered by \cite[Estimate (16), p.112]{Triebel} but, for the sake of completeness, we choose to provide the details.
\begin{proof}
Let $f\in \mathscr{S}(\R^N)$. Since translations $\tau_hf$ have Fourier transform $e^{ih\cdot\xi}\hat{f}$, the Fourier transform of $\Delta_h^Mf$ writes
\begin{align}
\mathcal{F}[\Delta_h^Mf](\xi)=\hat{f}(\xi)\,\sum_{j=0}^M\binom{M}{j}(-1)^{M-j}(e^{ih\cdot\xi})^j. \nonumber
\end{align}
And so, by applying the binomial formula and taking the inverse Fourier transform of the result one gets
\[\Delta_h^Mf=\mathcal{F}^{-1}[(e^{ih\cdot\xi}-1)^M\hat{f}].\]
Now let $h_1,h_2\in\R^N$ and $h=h_1+h_2$. Notice that we have
\[e^{ih\cdot\xi}-1=e^{ih_1\cdot\xi}(e^{ih_2\cdot\xi}-1)+e^{ih_1\cdot\xi}-1.\]
Let $P\in\mathbb{C}[X,Y]$ be the polynomial defined by
\[P(X,Y)=\left(X(Y-1)+(X-1)\right)^{2M}.\]
By the binomial formula one may find $Q_1,Q_2\in\mathbb{C}[X,Y]$ such that
\[P(X,Y)=(X-1)^MQ_1(X,Y)+X^M(Y-1)^MQ_2(X,Y).\]
This holds for any $X,Y\in\mathbb{C}$. In particular
\[(e^{ih\cdot\xi}-1)^{2M}=(e^{ih_1\cdot\xi}-1)^MQ_1(e^{ih_1\cdot\xi},e^{ih_2\cdot\xi})+e^{iMh_1\cdot\xi}(e^{ih_2\cdot\xi}-1)^MQ_2(e^{ih_1\cdot\xi},e^{ih_2\cdot\xi}).\]
Multiplying this equality by $\hat{f}(\xi)$ and taking the inverse Fourier transform of the result, we obtain:
\begin{align}
\Delta_h^{2M}f&=\mathcal{F}^{-1}\left[\sum_{k,\ell=0}^{M}\alpha_{k,\ell}(e^{ih_1\cdot\xi}-1)^M\mathcal{F}\big[f(\cdot+kh_1+\ell h_2)\big]\right] \nonumber \\
& \quad +\mathcal{F}^{-1}\left[\sum_{k,\ell=0}^{M}\beta_{k,\ell}e^{iMh_1\cdot\xi}(e^{ih_2\cdot\xi}-1)^M\mathcal{F}\big[f(\cdot+kh_1+\ell h_2)\big]\right] \nonumber
\end{align}
where $\alpha_{k,\ell}$ and $\beta_{k,\ell}$ are the respective coefficients of $Q_1$ and $Q_2$. Otherwise said,
\begin{align}
\Delta_h^{2M}f&=\sum_{k,\ell=0}^M\alpha_{k,\ell}\Delta_{h_1}^Mf\big(\cdot+kh_1+\ell h_2\big)+\sum_{k,\ell=0}^M\beta_{k,\ell}\Delta_{h_2}^Mf\big(\cdot+(k+M)h_1+\ell h_2\big). \nonumber
\end{align}
We therefore obtain that, for each $f\in\mathscr{S}(\R^N)$
\begin{align}
\|\Delta_h^{2M}f\|_{L^p(\mathbb{R}^N)}\leq C\left(\|\Delta_{h_1}^{M}f\|_{L^p(\mathbb{R}^N)}+\|\Delta_{h_2}^{M}f\|_{L^p(\mathbb{R}^N)}\right), \label{iteratedest}
\end{align}
for some constant $C>0$ depending only on $M$, $Q_1$ and $Q_2$. Since $\mathscr{S}(\R^N)$ is dense in $L^p(\R^N)$ for $p<\infty$ the result follows for every $f\in L^p(\R^N)$. When $p=\infty$, the above still holds in the $\mathscr{S}'$ sense and, thus, extends to $L^\infty(\R^N)$ as well.
\end{proof}
Also, we recall the following well-known fact.
\begin{lemma}\label{2Morder}
Let $M\in\N^\ast$, $s\in(0,M)$ and $f\in L^p(\R^N)$. Then,
\begin{align}
\sup_{h\ne0}\frac{\|\Delta_h^Mf\|_{L^p(\R^N)}}{|h|^s}\leq C(s,M)\sup_{h\ne0}\frac{\|\Delta_h^{2M}f\|_{L^p(\R^N)}}{|h|^s}, \label{supsup}
\end{align}
for some constant $C(s,M)>0$ depending only on $s$ and $M$. Similarly,
\begin{align}
\limsup_{|h|\to0}\frac{\|\Delta_h^Mf\|_{L^p(\R^N)}}{|h|^s}\leq C(s,M)\limsup_{|h|\to0}\frac{\|\Delta_h^{2M}f\|_{L^p(\R^N)}}{|h|^s}. \label{lemmeSUP}
\end{align}
\end{lemma}
This is a consequence of \cite[Estimate (45), p.99]{Triebel}, but the proof being very short we chose to provide all the details.
\begin{proof}
Let $f\in L^p(\R^N)$ and $P\in\mathbb{C}[X]$ be the unique polynomial such that
\begin{align}
P(X)(X-1)=1-\left(\frac{X+1}{2}\right)^M. \label{polynome}
\end{align}
Note that $P$ exists because $X-1$ is a divisor of the right-hand side of \eqref{polynome}. In particular, we have that
\begin{align}
(X-1)^M=\frac{1}{2^M}(X^2-1)^M+(X-1)^{M+1}P(X).
\end{align}
Hence, for every $h,\xi\in\R^N$ we have
\begin{align}
(e^{ih\cdot\xi}-1)^M=\frac{1}{2^M}(e^{i2h\cdot\xi}-1)^M+(e^{ih\cdot\xi}-1)^{M+1}P(e^{ih\cdot\xi}).
\end{align}
Whence, reasoning as in Lemma \ref{iteratedtrans}, we obtain
\begin{align}
\Delta_h^Mf(x)=\frac{1}{2^M}\Delta_{2h}^Mf(x)+\Delta_h^{M+1}\bigg(\sum_{\ell\in L}~a_\ell f(x+h\ell)\bigg),
\end{align}
for some finite set of indices $L\subset\N$ and coefficients $a_\ell$ depending on $P$. Thus, for every $s\in(0,M)$, $h\ne0$ and $f\in L^p(\R^N)$ it holds,
\begin{align}
\frac{\|\Delta_h^Mf\|_{L^p(\R^N)}}{|h|^s}\leq \frac{1}{2^{M-s}}~\frac{\|\Delta_{2h}^Mf\|_{L^p(\R^N)}}{|2h|^s}+C~\frac{\|\Delta_h^{M+1}f\|_{L^p(\R^N)}}{|h|^s}.
\end{align}
We obtain that
\begin{align}
\left(1-\frac{1}{2^{M-s}}\right)\sup_{h\ne0}\frac{\|\Delta_h^Mf\|_{L^p(\R^N)}}{|h|^s}\leq C~\sup_{h\ne0}\frac{\|\Delta_h^{M+1}f\|_{L^p(\R^N)}}{|h|^s}.
\end{align}
Therefore \eqref{supsup} follows by induction. The proof of \eqref{lemmeSUP} is similar.
\end{proof}
~~\\

\subsection{Proof of Theorem \ref{THEO}} ~~ \\

Let $M\in\N^\ast$, $s>0$, $p\in[1,\infty]$, $\omega\in C_{\mathrm{inc}}^+$ and $(\rho_\eps)_{\eps>0}$ be as in the statement of Theorem \ref{THEO}. Here again, we will make use of the short notation
\begin{align}
\mathscr{D}_\omega(\rho_\eps,f):=\int_{\R^N}\rho_\eps(h)~\omega\left(\frac{\|\Delta_h^Mf\|_{L^p(\R^N)}}{|h|^s}\right)\mathrm{d}h.
\end{align}
In addition, we call $\mathfrak{M}(\mathbb{R}^N)$ the set of mollifiers $(\rho_\eps)_{\eps>0}\subset L^1(\R^N)$ satisfying \eqref{molli3} for some $\rho\in L^1(\R^N)$ such that there exists a number $\delta>0$ and a nonnegative, nondecreasing, radial function $\Psi\in C(\R^N)$ with
\begin{align}
\rho(h)\geq \Psi(h) \quad \text{ for a.e. }~~h\in B_\delta \quad \text{  and  } \quad \int_{B_{\delta/4}}\Psi>0. \label{gene}
\end{align}
\noindent Note that, by Proposition \ref{cotefacile}, we only need to establish a one-sided inequality.

We begin with a few preliminary facts (Claim A and Claim B) showing that the proof of Theorem \ref{THEO} reduces to the case where $(\rho_\eps)_{\eps>0}\in\mathfrak{M}(\R^N)$.\\

\noindent \textbf{Claim A.} \emph{It is enough to establish Theorem \ref{THEO} for radial} $\rho$\emph{'s} \emph{such that}
\begin{align}
\underset{\mathscr{A}}{\mathrm{ess}\,\mathrm{inf}}\,\rho>0 \qquad{\mbox{ \emph{for some annulus} }}~~\mathscr{A}\subset\R^N~~\mbox{\emph{centered at zero}}. \label{cdnA}
\end{align}
\begin{proof}[Proof of Claim A]
Let $\rho\in L^1(\R^N)$ be a nonnegative radial function with unit mass. Then, there is a nonnegative function $\tilde{\rho}\in L_{loc}^1(\R_+)$ with $\rho(z)=\tilde{\rho}(|z|)$. In particular, we may find some $0<c_1<c_2$ such that
$$ \int_{c_1}^{c_2}\tilde{\rho}(\theta)\mathrm{d}\theta>0. $$
Let $0<\theta_0<1$ be such that $c_1>c_2\theta_0$ and let $\rho^*$ be the radial function given by
$$ \rho^*(z):=C\fint_{\theta_0}^1\rho(\theta z)\mathrm{d}\theta=C\fint_{\theta_0}^1\tilde{\rho}(\theta |z|)\mathrm{d}\theta \qquad{\mbox{ for }}~~z\in\R^N, $$
where $C>0$ is given by
$$ C:=(1-\theta_0)\left(\int_{\theta_0}^1\frac{\mathrm{d}\theta}{\theta^N}\right)^{-1}. $$
Notice that, by Fubini, $\rho^*\in L^1(\R^N)$ and $\rho^*$ has unit mass. Indeed, this is because
\begin{align}
\|\rho^*\|_{L^1(\R^N)}=\frac{C}{1-\theta_0}\int_{\theta_0}^1\int_{\R^N}\rho(\theta z)\mathrm{d}z\mathrm{d}\theta=\frac{C}{1-\theta_0}\int_{\theta_0}^1\frac{\mathrm{d}\theta}{\theta^N}=1. \nonumber
\end{align}
Furthermore, one easily checks that $\rho^*$ satisfies \eqref{cdnA}. Indeed, we have
\begin{align}
\underset{|z|\in\left[c_2,\frac{c_1}{\theta_0}\right]}{\mathrm{ess}\,\mathrm{inf}}\,\,\rho^*(z)&=\frac{C}{1-\theta_0}\,\,\underset{|z|\in\left[c_2,\frac{c_1}{\theta_0}\right]}{\mathrm{ess}\,\mathrm{inf}}\,\int_{\theta_0|z|}^{|z|}\tilde{\rho}(\theta)\frac{\mathrm{d}\theta}{|z|}\geq \frac{C\theta_0}{c_1(1-\theta_0)}\int_{c_1}^{c_2}\tilde{\rho}(\theta)\mathrm{d}\theta>0. \nonumber
\end{align}
On the other hand, we have
\begin{align}
\rho_\eps(\theta\,\cdot)=\theta^{-N}\rho_{\eps/\theta}\leq\theta_0^{-N}\rho_{\eps/\theta} \qquad{\mbox{ for any }}~~\theta\in[\theta_0,1].  \nonumber
\end{align}
Hence,
\begin{align}
\frac{1}{C}\,\mathscr{D}_\omega(\rho_\eps^*,f)=\fint_{\theta_0}^1\mathscr{D}_\omega(\rho_\eps(\theta\,\cdot),f)\mathrm{d}\theta\leq \theta_0^{-N}\sup_{\theta_0\leq\theta\leq1}\,\mathscr{D}_\omega(\rho_{\eps/\theta},f). \label{ClaimA-limsup}
\end{align}
Assuming that Theorem \ref{THEO} holds for mollifiers with $\rho$ satisfying \eqref{cdnA}, we finally obtain
\begin{align}
\omega\left([f]_{B_{p,\infty}^s(\R^N)}\right)\lesssim\, \sup_{\eps>0}\,\mathscr{D}_\omega(\rho_\eps,f).  \nonumber
\end{align}
Thus, the claim follows.
\end{proof}

\noindent \textbf{Claim B.} \emph{It is enough to establish Theorem \ref{THEO} for} $(\rho_\eps)_{\eps>0}\in\mathfrak{M}(\R^N)$.

\begin{proof}[Proof of Claim B] Let $\rho\in L^1(\R^N)$ be a nonnegative radial function with unit mass. Then, there is a nonnegative function $\tilde{\rho}\in L_{loc}^1(\R_+)$ with $\rho(z)=\tilde{\rho}(|z|)$. On account of Claim A, we may assume that there are some $0<r_1<r_2$ and some $\alpha>0$ with
\begin{align}
\tilde{\rho}(t)\geq \alpha~\mathds{1}_{(r_1,r_2)}(t)=:\Psi(t) \qquad{\mbox{ for a.e. }}~~t\geq0.  \nonumber %\label{annnnulus}
\end{align}
If $r_1<\frac{r_2}{4}$, then $(\rho_\eps)_{\eps>0}\in\mathfrak{M}(\R^N)$ and the claim is trivial. Hence, we may assume that $r_1\geq\frac{r_2}{4}$. To show that the latter case reduces to the former, we simply clip together rescaled copies of $\Psi$ as follows. For any $j\geq0$, define
$$ \theta_j:=\left(\frac{r_1}{r_2}\right)^j \quad \text{  and  } \quad \Psi_{\theta_j}(t):=\theta_j^{-N}\Psi\left(\frac{t}{\theta_j}\right). $$
Observe that, by construction, $\theta_j\to0$ as $j\to\infty$ and
$$ 0<\cdots<\theta_{j+1}r_1<\theta_{j+1}r_2=\theta_jr_1<\theta_jr_2<\cdots<\theta_1r_2=r_1<r_2. $$
Thus, the supports of the $\Psi_{\theta_j}$'s are mutually disjoint and they form a countable partition of $[0,r_2]$.
Now take an integer $k\in\N$ such that
\begin{align}
k>\frac{\ln\left(\frac{1}{5}\right)}{\ln\left(\frac{r_1}{r_2}\right)}.  \nonumber
\end{align}
By construction, this guarantees that $\theta_k<\frac{1}{5}$ and, in turn, that
$$ \mathrm{supp}(\Psi_{\theta_k})\subset \left[0,\frac{r_2}{5}\right]. $$
In particular, we have
$$ \bigg[\frac{r_2}{5},\,r_2\bigg)\subset\bigcup_{j=0}^{k}\mathrm{supp}(\Psi_{\theta_j}). $$
Fix such a $k\in\N$ and set $J=[\![0,k]\!]$. Then, the function
\begin{align}
\eta^\ast(t):=\sum_{j\in J}\Psi_{\theta_j}(t), \qquad{\mbox{ for }}~~t\geq0,  \nonumber
\end{align}
is bounded and
$$ \bigg[\frac{r_2}{5},\,r_2\bigg)\subset \mathrm{supp}(\eta^*) \subset [0,r_2]. $$
Moreover, $\eta^*$ satisfies the following monotonicity property:
\begin{align}
\eta^\ast(t_1)\geq \eta^\ast(t_2)\geq\alpha \quad \text{ whenever } \quad \frac{r_2}{5}<t_1<t_2<r_2.  \nonumber
\end{align}
Thus, there is a nondecreasing function $\Phi^*\in C(\R_+)$ with
\begin{align}
\eta^\ast\geq \Phi^* \quad \text{ a.e. in }~~[0,r_2] \quad \text{ and } \quad \int_0^{r_2/4}\Phi^*(t)\,t^{N-1}\mathrm{d}t>0. \label{K-1D}
\end{align}
\begin{figure}
\centering
\includegraphics[scale=0.57]{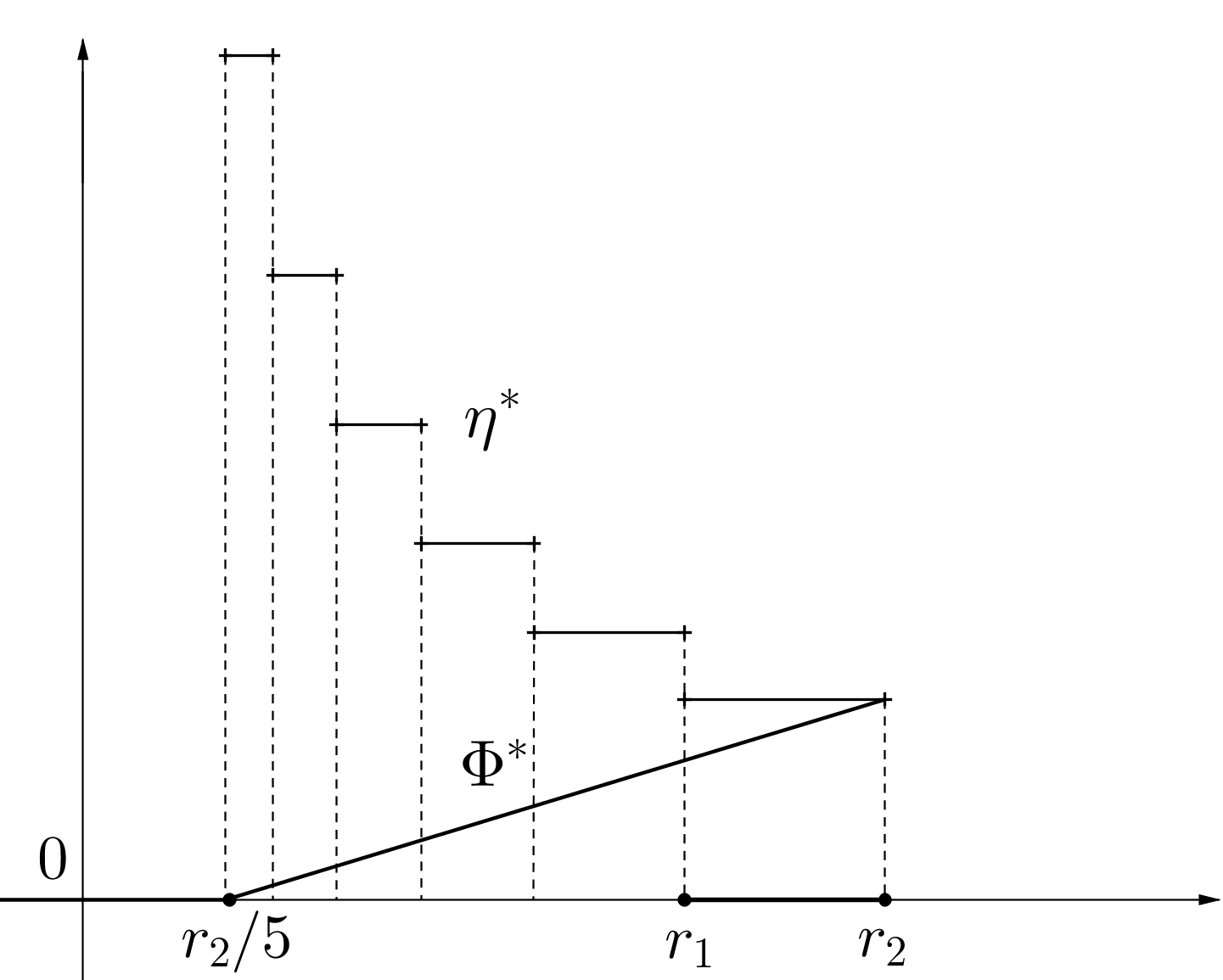} \\
\caption{Construction of $\eta^\ast$ and $\Phi^*$.} \label{fig1}
\end{figure}
%\begin{center}
%\includegraphics[scale=0.57]{clipp2x.png} \\
%\textsc{Fig. 1} : \emph{Construction of} $\eta^\ast$ \emph{and} $\Phi^*$. \\
%\end{center}
Indeed, it suffices to take e.g.
\begin{align}
\Phi^*(t):=\frac{5\alpha}{4r_2}\bigg(t-\frac{r_2}{5}\bigg)~\mathds{1}_{(\frac{r_2}{5},\infty)}(t).  \nonumber
\end{align}
See Figure \ref{fig1} for a visual evidence. Now, set
\begin{align}
\Phi(x):=\frac{1}{c}\,\Phi^*(|x|) \quad \text{ and } \quad \eta(x):=\frac{1}{c}\,\eta^\ast(|x|)  \quad \text{  where  } \quad c:=\int_{B_{r_2}}\eta^\ast(|x|)\mathrm{d}x.  \nonumber
\end{align}
By construction, $\eta\in L^1(\R^N)$ and $\eta$ has unit mass. Moreover, by \eqref{K-1D} we have
$$ \eta\geq\Phi  \quad \text{ a.e. in }~~B_{r_2} \quad \text{ and } \quad \int_{B_{r_2/4}}\Phi>0. $$
Whence, $(\eta_\eps)_{\eps>0}\in\mathfrak{M}(\R^N)$. On the other hand,
\begin{align}
c~\mathscr{D}_\omega(\eta_\eps,f)&=\sum_{j\in J}\mathscr{D}_\omega\big(\Psi_{\theta_j\eps}(\left|\cdot\right|),f\big)\leq \sum_{j\in J} \mathscr{D}_\omega(\rho_{\theta_j\eps},f), \label{ClaimB-limsup}
\end{align}
Hence, one obtains
\begin{align}
\sup_{\eps>0}\,\mathscr{D}_\omega(\eta_\eps,f)\leq \frac{\# J}{c}\,\sup_{\eps>0}\,\mathscr{D}_\omega(\rho_{\eps},f). \nonumber
\end{align}
Assuming that Theorem \ref{THEO} holds for mollifiers belonging to $\mathfrak{M}(\R^N)$, we finally obtain
\begin{align}
\omega\left([f]_{B_{p,\infty}^s(\R^N)}\right)\lesssim\, \sup_{\eps>0}\,\mathscr{D}_\omega(\rho_\eps,f). \nonumber
\end{align}
Thus, the claim follows.
\end{proof}
\begin{rem}\label{ClaimAB-limsup}
By \eqref{ClaimA-limsup} and \eqref{ClaimB-limsup} we also have that
$$ \omega\left(\limsup_{|h|\to0}\frac{\|\Delta_h^{M}f\|_{L^p(\R^N)}}{|h|^{s}}\right)\lesssim\limsup_{\eps\downarrow0}\,\mathscr{D}_\omega(\rho_\eps,f), $$
holds for any $(\rho_\eps)_{\eps>0}$ satisfying \eqref{molli} and \eqref{molli3} whenever it holds for any $(\rho_\eps)_{\eps>0}$ belonging to $\mathfrak{M}(\R^N)$.
\end{rem}
We may now complete the proof of Theorem \ref{THEO}.
\vskip 0.3cm

\noindent \textsc{Step 1: case $M=1$ and $s\in(0,1)$.}
\vskip 0.3cm
Let $p\in[1,\infty]$, $(\rho_\eps)_{\eps>0}\in\mathfrak{M}(\R^N)$, $\omega\in C_{\mathrm{inc}}^+$ and $f\in L^p(\R^N)$ satisfying \eqref{LAMBDA}. Let $h\in \R^N$ (to be fixed later) and let $z\in\R^N$.
Then, we have
\begin{align}
\tau_z f-f=\tau_{h}f-f+\tau_{h}(\tau_{z-h}f-f). \label{labazz}
\end{align}
This implies
\begin{align}
\|\tau_hf-f\|_{L^p(\R^N)}\leq\|\tau_z f-f\|_{L^p(\R^N)}+\|\tau_{z-h}f-f\|_{L^p(\R^N)}. \label{estimationtrans}
\end{align}
Now, choose $z\in B_{|h|}(h)$. Then, since $z$ and $z-h$ belong to $B_{2|h|}$, it comes
\begin{align}
\frac{1}{2^s}\frac{\|\tau_hf-f\|_{L^p(\R^N)}}{|h|^s}\leq\frac{\|\tau_z f-f\|_{L^p(\R^N)}}{|z|^s}+\frac{\|\tau_{z-h}f-f\|_{L^p(\R^N)}}{|z-h|^s}. \label{stab}
\end{align}
Since $\omega$ is roughly subadditive, there exists a constant $A_\omega>0$ depending only on $\omega$ and such that, for every $x,y\in\R_+$,
\begin{align}
\omega(x+y)\leq A_\omega\left\{\omega(x)+\omega(y)\right\}. \label{subdefOmega}
\end{align}
\begin{rem}\label{puissanceOmegaSubad}
Note that \eqref{subdefOmega} implies $\omega(2x)\leq 2A_\omega \omega(x)$ and, since $\omega\in C_{\mathrm{inc}}^+$, it is increasing, thus $\omega(2^sx)\leq 2A_\omega \omega(x)$ for $s\leq 1$. Similarly, when $s\leq M\in\N^*$, one has $\omega(2^sx)\leq (2A_\omega)^M\omega(x)$.
\end{rem}
\noindent From \eqref{stab}, Remark \ref{puissanceOmegaSubad} and thanks to $s\leq 1$, using the short notation $\Delta_h^1f=\tau_hf-f$ we have
\begin{align}
\omega\left(\frac{\|\Delta_h^1f\|_{L^p(\R^N)}}{|h|^s}\right)&\leq \omega\left(2^s\frac{\|\Delta_z^1f\|_{L^p(\R^N)}}{|z|^s}+2^s\frac{\|\Delta_{z-h}^1f\|_{L^p(\R^N)}}{|z-h|^s}\right) \nonumber \\
&\leq A_\omega\left\{\omega\left(2^s\frac{\|\Delta_z^1f\|_{L^p(\R^N)}}{|z|^s}\right)+\omega\left(2^s\frac{\|\Delta_{z-h}^1f\|_{L^p(\R^N)}}{|z-h|^s}\right)\right\} \nonumber \\
&\leq 2A_\omega^2\left\{\omega\left(\frac{\|\Delta_z^1f\|_{L^p(\R^N)}}{|z|^s}\right)+\omega\left(\frac{\|\Delta_{z-h}^1f\|_{L^p(\R^N)}}{|z-h|^s}\right)\right\}. \label{diffinie}
\end{align}
Using $(\rho_\varepsilon)_{\varepsilon>0}\in\mathfrak{M}(\R^N)$ we know there is a radially nondecreasing $\Psi\in C(\R^N)$ and a number $\delta>0$ such that
\begin{align}
\rho_\varepsilon(z)\geq \Psi_\varepsilon(|z|), \quad \text{ for a.e.}~~z\in B_{\varepsilon\delta}~~\text{ and all }~~\eps>0. \label{rq1}
\end{align}
As seen in Figure \ref{fig2}, we clearly have
\begin{align}
\Psi_\varepsilon(|z-h|)\leq\Psi_\varepsilon(|z|), \quad \text{ for all}~~h\in B_{\varepsilon\delta/2},~~z\in B_{|h|/2}(h)~~\text{ and }~~\eps>0. \label{rq2}
\end{align}
Let $h\in B_{\varepsilon\delta/2}$. Multiplying \eqref{diffinie} by $\Psi_\varepsilon(|z-h|)$ and using \eqref{rq1}-\eqref{rq2} we obtain
\begin{align}
\Psi_\varepsilon&(|z-h|)~\omega\left(\frac{\|\Delta_{h}^1f\|_{L^p(\R^N)}}{|h|^s}\right) \nonumber \\
&\leq 2A_\omega^2\left\{\rho_\varepsilon(z)~\omega\left(\frac{\|\Delta_{z}^1f\|_{L^p(\R^N)}}{|z|^s}\right)+\rho_\varepsilon(z-h)~\omega\left(\frac{\|\Delta_{z-h}^1f\|_{L^p(\R^N)}}{|z-h|^s}\right)\right\}, \nonumber
\end{align}
and this holds for all $h\in B_{\varepsilon\delta/2}$ and a.e. $z\in B_{|h|/2}(h)$.
\begin{figure}
\centering
\includegraphics[scale=0.35]{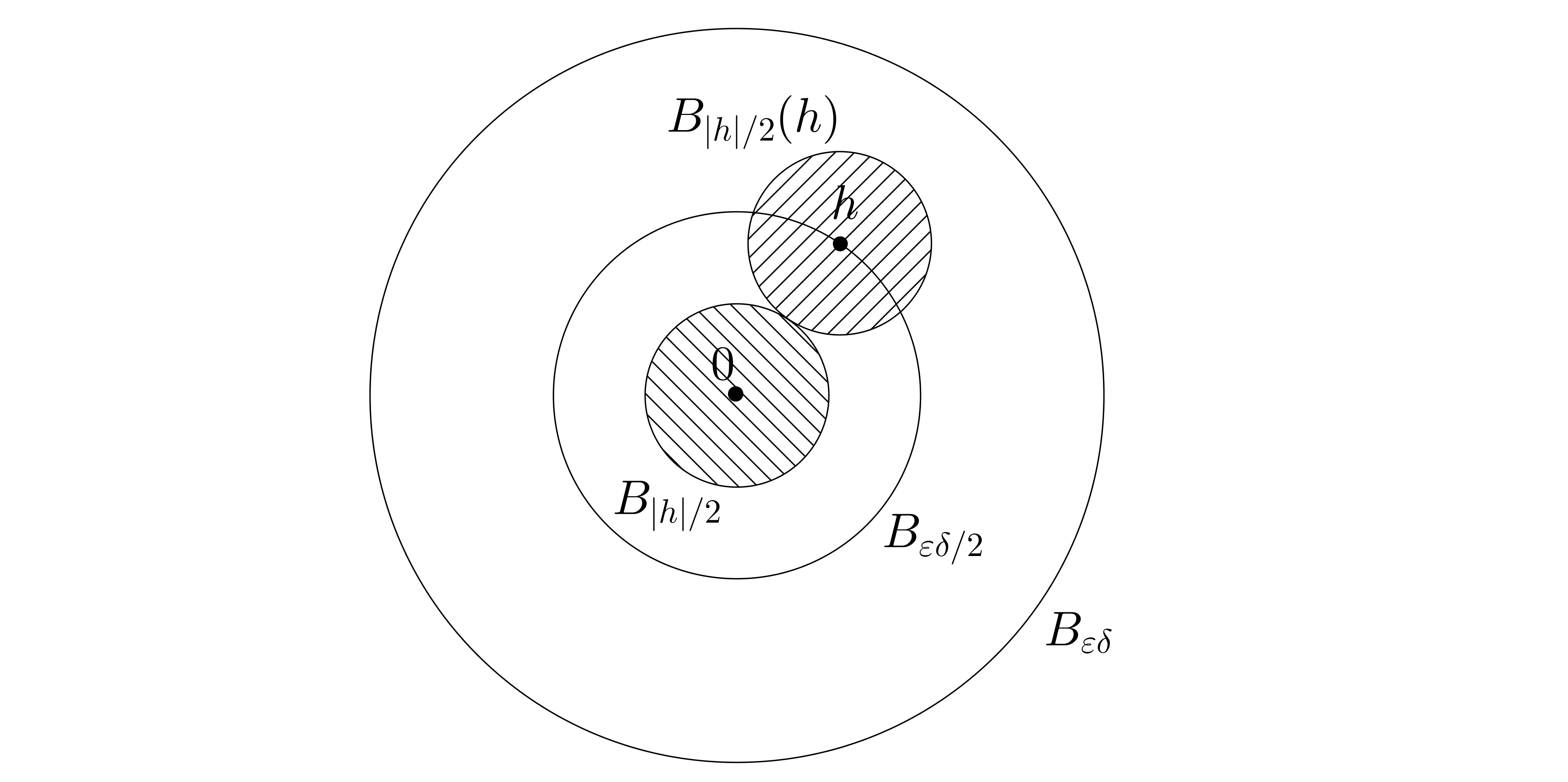} \\
\caption{Spatial conditions on $h$ and $z$.} \label{fig2}
\end{figure}
%\begin{center}
%\includegraphics[scale=0.35]{mol3.png} \\
%\textsc{Fig. 2} : \emph{}. \\
%~~ \\
%\end{center}
\noindent So, taking $|h|=\delta\varepsilon/2$ and integrating over $z\in B_{|h|/2}(h)$, yields:
\begin{align}
C(\Psi,\delta)~\omega\left(\frac{\|\Delta_{h}^1f\|_{L^p(\R^N)}}{|h|^{s}}\right)&\leq 2A_\omega^2\int_{B_{|h|/2}(h)}\rho_\varepsilon(z)~\omega\left(\frac{\|\Delta_{z}^1f\|_{L^p(\R^N)}}{|z|^{s}}\right)\mathrm{d}z \nonumber \\
&~~+2A_\omega^2\int_{B_{|h|/2}(h)}\rho_\varepsilon(z-h)~\omega\left(\frac{\|\Delta_{z-h}^1f\|_{L^p(\R^N)}}{|z-h|^{s}}\right)\mathrm{d}z \nonumber \\
& \leq 4A_\omega^2\int_{\R^N}\rho_\varepsilon(z)~\omega\left(\frac{\|\Delta_{z}^1f\|_{L^p(\R^N)}}{|z|^{s}}\right)\mathrm{d}z, \nonumber
\end{align}
where
\[C(\Psi,\delta):=\int_{B_{|h|/2}(h)}\Psi_\varepsilon(|z-h|)\mathrm{d}z=\int_{B_{\delta/4}}\Psi(|z|)\mathrm{d}z>0.\]
Whence,
\begin{align}
\omega\left(\frac{\|\Delta_{h}^1f\|_{L^p(\R^N)}}{|h|^{s}}\right)\leq \frac{4A_\omega^2}{C(\Psi,\delta)}\int_{\R^N}\rho_\varepsilon(z)~\omega\left(\frac{\|\Delta_{z}^1f\|_{L^p(\R^N)}}{|z|^{s}}\right)\mathrm{d}z. \label{controle}
\end{align}
Passing to the limit superior as $|h|\to0$ in \eqref{controle} it follows
\[\omega\big([f]_{N^{s,p}(\R^N)}\big)\leq \frac{4A_\omega^2}{C(\Psi,\delta)}\limsup_{\varepsilon\downarrow0}\int_{\R^N}\rho_\varepsilon(z)~\omega\left(\frac{\|\Delta_{z}^1f\|_{L^p(\R^N)}}{|z|^{s}}\right)\mathrm{d}z, \]
where we used the continuity of $\omega$. This, together with Proposition \ref{cotefacile} yields
\begin{align}
\omega\big([f]_{N^{s,p}(\R^N)}\big)\sim\limsup_{\varepsilon\downarrow0}\int_{\R^N}\rho_\varepsilon(z)~\omega\left(\frac{\|\Delta_{z}^1f\|_{L^p(\R^N)}}{|z|^{s}}\right)\mathrm{d}z. \label{limsup}
\end{align}
Similarly, taking the supremum over $h\ne0$ in \eqref{controle}, we obtain
$$ \omega\left([f]_{B_{p,\infty}^s(\R^N)}\right)\sim\,\sup_{\eps>0}\,\int_{\R^N}\rho_\eps(h)~\omega\left(\frac{\|\Delta_h^1f\|_{L^p(\R^N)}}{|h|^s}\right)\mathrm{d}h.$$
\begin{rem}
Estimate \eqref{limsup} together with Proposition \ref{completion} and Remark \ref{ClaimAB-limsup} prove Corollary \ref{approx} for $1\leq p<\infty$ and $s\in(0,1)$ (recall we have assumed $\omega(0)=0$).
\end{rem}
\vskip 0.3cm

\noindent \textsc{Step 2: case $M\geq2$ and $s\in(0,M)$.}
\vskip 0.3cm
The assumption $s\in(0,1)$ being artificial (by Remark \ref{puissanceOmegaSubad}) the above still holds for general $s>0$. In particular, replacing \eqref{estimationtrans} by the estimate of Lemma \ref{iteratedtrans}, for $f\in L^p(\R^N)$, one obtains
\begin{align}
\omega\left(\frac{\|\Delta_h^{2M}f\|_{L^p(\R^N)}}{|h|^{s}}\right)\leq C(M,\rho,\omega)\int_{\R^N}\rho_\varepsilon(z)~\omega\left(\frac{\|\Delta_z^Mf\|_{L^p(\R^N)}}{|z|^{sp}}\right)\mathrm{d}z, \label{2M2M}
\end{align}
for any $s\in(0,M)$. Taking the supremum over $\eps>0$ (i.e. over $|h|>0$) and recalling that $\omega$ is a continuous, non-decreasing function, we find that
\[\omega\left([f]_{B_{p,\infty}^s(\R^N)}\right)\lesssim\,\sup_{\eps>0}\int_{\R^N}\rho_\varepsilon(z)~\omega\left(\frac{\|\Delta_z^Mf\|_{L^p(\R^N)}}{|z|^{s}}\right)\mathrm{d}z.\]
This is because the space $B_{p,\infty}^s(\R^N)$ with $s\in(0,M)$ is characterized by finite differences of order $M$, i.e.
\[[f]_{B_{p,\infty}^s(\R^N)}\sim \sup_{|h|\ne0}\frac{\|\Delta_h^{2M}f\|_{L^p(\R^N)}}{|h|^s}, \quad \forall s\in(0,M),\]
Indeed, recall Lemma \ref{2Morder} and $\|\Delta_h^{2M}f\|_{L^p(\R^N)}\leq C(M)\|\Delta_h^Mf\|_{L^p(\R^N)}$.

%We begin with the first claim of Remark \ref{rkhypo}. Assume $\rho\in L^1(\R^N)$ is locally bounded from below around the origin by a radial function, i.e. that there exists some number $r>0$ and a nonnegative, radial function $\varphi^\ast\in L^1(B_r)$ with $\rho\geq\varphi^\ast$ a.e. in $B_r$ and $\int_{B_r}\varphi^\ast>0$. Then, if we let $\varphi:=\varphi^\ast/(\int_{B_r}\varphi^\ast)$ we have
%\begin{align}
%A_\omega(\varphi_\eps,f)\leq \|\varphi^\ast\|_{L^1(B_r)}A_\omega(\rho_\eps,f),~~\text{for all}~~\eps>0,
%\end{align}
%where $\varphi_\eps:=\eps^{-N}\varphi(\frac{\cdot}{\eps})$ defines a sequence of radial mollifiers. By \textsc{Steps 1} and 2 we obtain the first claim of Remark \ref{rkhypo}. The latter being straightforward, it remains to clarify the second one. \\
\begin{rem}
As above, we still have
$$ \omega\left(\limsup_{|h|\to0}\frac{\|\Delta_h^{M}f\|_{L^p(\R^N)}}{|h|^{s}}\right)\sim\limsup_{\eps\downarrow0}\,\mathscr{D}_\omega(\rho_\eps,f). $$
So that Corollary \ref{approx} follows in that case too.
\end{rem}
\begin{rem}\label{rkbest}
Note that, when $(\rho_\eps)_{\eps>0}\in \mathfrak{M}(\R^N)$ (with corresponding $\Psi$ and $\delta$), we have actually proved a stronger estimate than needed. Indeed, we have shown that for any $h\in\R^N\setminus\{0\}$, $s\in(0,1]$, $p\in[1,\infty]$ and $(\rho_\eps)_{\eps>0}\in \mathfrak{M}(\R^N)$ it holds
\begin{align}
\omega\left(\frac{\|\Delta_h^1f\|_{L^p(\R^N)}}{|h|^s}\right)\leq C(\Psi,\delta,A_\omega)\int_{\R^N}\rho_{\eps(|h|)}(z)~\omega\left(\frac{\|\Delta_z^1f\|_{L^p(\R^N)}}{|z|^s}\right)\mathrm{d}z,
\end{align}
where $\eps(t)=\frac{2t}{\delta}$ and $A_\omega$ is as in Definition \ref{rsubad}.
\end{rem}
\vskip 0.3cm

\noindent \textsc{Step 3: proof of Remark \ref{rkhypo}.}
\vskip 0.3cm
Let $1\leq p<\infty$, $\omega\in C_{\mathrm{inc}}^+$ and $\Omega\in C_{\mathrm{inc}}$ satisfying \eqref{OOmg1}. Then, we have
\begin{align}
\omega\left(\int_{\R^N}\Omega\left(\frac{|\Delta_h^Mf(x)|}{|h|^s}\right)\mathrm{d}x\right)&\geq \omega\left(m_1\frac{\|\Delta_h^Mf\|_{L^p(\R^N)}^p}{|h|^{sp}}\right) \nonumber \\
&\geq K_1(m_1,A_\omega)\,\omega\left(\frac{\|\Delta_h^Mf\|_{L^p(\R^N)}^p}{|h|^{sp}}\right),
\end{align}
where $K_1(m_1,A_\omega)>0$ and $A_\omega>0$ is a number such that $\omega$ satisfies the condition of Definition \ref{rsubad} with $A=A_\omega$. Similarly, for some $K_2(m_2,A_\omega)>0$,
\begin{align}
\omega\left(\int_{\R^N}\Omega\left(\frac{|\Delta_h^Mf(x)|}{|h|^s}\right)\mathrm{d}x\right)\leq K_2(m_2,A_\omega)\,\omega\left(\frac{\|\Delta_h^Mf\|_{L^p(\R^N)}^p}{|h|^{sp}}\right).
\end{align}
Now, since $\tilde{\omega}=\omega\circ\left|\cdot\right|^p$ lies in $C_{\mathrm{inc}}^+$ (by Remark \ref{cinc}) we obtain the desired claim, i.e. that
$$ \omega\left([f]_{B_{p,\infty}^s(\R^N)}^p\right)\sim\,\sup_{\eps>0}\,\int_{\R^N}\rho_\eps(h)\,\omega\left(\int_{\R^N}\Omega\left(\frac{|\Delta_h^Mf(x)|}{|h|^s}\right)\mathrm{d}x\right)\mathrm{d}h. $$
The remaining claims of Remark \ref{rkhypo} follow by a similar argument of comparison. \\

\section{Characterization of Sobolev and $BV$ spaces} \label{sectionSOBOLEV}

We begin with a preliminary result.
\begin{lemma}\label{limEsup}
Let $p\in[1,\infty]$, $f\in L^p(\R^N)$ and let
\begin{align}
A:=\sup_{h\ne0}\frac{\|\Delta_h^1f\|_{L^p(\R^N)}}{|h|}.
\end{align}
If $A$ is finite, then,
\begin{align}
A=\limsup_{|h|\to0}\frac{\|\Delta_h^1f\|_{L^p(\R^N)}}{|h|}.
\end{align}
\end{lemma}
Although our argument is much simpler, a proof of a similar result (involving moduli of continuity) may be found in \cite{Triebel2}. However, the argument in \cite{Triebel2} heavily rely on the continuity of $\|\Delta_h^1f\|_{L^p(\R^N)}$ and, thus, does not cover the case $p=\infty$. We show that, in fact, it is enough to ask only for some kind of subadditivity.
\begin{proof}
Let $f\in L^p(\R^N)$, $1\leq p\leq\infty$. For any $t\in\R$, define
\begin{align}
F(t):=\sup_{\sigma\in\S^{N-1}}\|\Delta_{\sigma t}^1f\|_{L^p(\mathbb{R}^N)}.
\end{align}
Clearly, $F$ is measurable. Now, let $t_1,t_2\in\R$. Specializing \eqref{estimationtrans} in $h=(t_1+t_2)\sigma$ and $z=t_1\sigma$, for some $\sigma\in\S^{N-1}$, yields
$$ \|\Delta_{\sigma(t_1+t_2)}^1f\|_{L^p(\mathbb{R}^N)}\leq \|\Delta_{\sigma t_1}^1f\|_{L^p(\mathbb{R}^N)}+\|\Delta_{-\sigma t_2}^1f\|_{L^p(\mathbb{R}^N)}\leq F(t_1)+F(t_2). $$
Consequently,
\begin{align}
F(t_1+t_2)\leq F(t_1)+F(t_2) \qquad{\mbox{ for all }}~~t_1,t_2\in\R.
\end{align}
Whence, $F:\R\to[0,\infty)$ is a measurable, subadditive function. Now suppose that
\begin{align}
A:=\sup_{t>0}\frac{F(t)}{t}<\infty.
\end{align}
Then, by the limit theorem of subadditive functions \cite[Theorem 16.3.3, p.467]{Kuczma},
\begin{align}
A=\lim_{t\downarrow0}\frac{F(t)}{t}=\lim_{t\downarrow0}\sup_{\sigma\in\S^{N-1}}\frac{\|\Delta_{\sigma t}^1f\|_{L^p(\mathbb{R}^N)}}{t}.
\end{align}
This proves the lemma.
\end{proof}
$~~$
\subsection{Proofs of Theorems \ref{THEO2} and \ref{THEO3}} ~~ \\

\begin{proof} The proof follows from a straightforward adaptation of the proof of Theorem \ref{THEO} in the case $M=1$ and $s\in(0,1)$ with Lemma \ref{limEsup} and the fact that, using for example \cite[Theorem 2, Theorem 3']{BBM},
\begin{align}
\|\nabla f\|_{L^p(\R^N)}\sim \limsup_{|h|\to0}~\frac{\|f(\cdot+h)-f\|_{L^p(\R^N)}}{|h|}, \label{nnnnbla}
\end{align}
for all $1\leq p<\infty$ (when $p=1$ the left-hand side of \eqref{nnnnbla} is to be understood in the $BV$-sense, i.e. as the total mass of the Radon measure $\nabla f$). \\

The case $p=\infty$ follows from the arguments above and the definition of the Lipschitz semi-norm.
\end{proof}
$~~$
\subsection{A limiting embedding between Lipschitz and Besov spaces} ~~ \\

This subsection is devoted to the proof of Theorem \ref{imbLip}. To this end, we recall the following improvement of the Chebychev inequality due to Bourgain, Brezis and Mironescu \cite{BBM}.
\begin{lemma}[Bourgain, Brezis, Mironescu, \cite{BBM}]\label{chebychev}
 Let $g,h:(0,\delta)\to\R_+$. Assume that $g(t)\leq g(t/2)$ for all $t\in(0,\delta)$ and that $h$ is non-increasing. Then, for some constant $C=C(N)>0$,
\begin{align}
\delta^{-N}\int_0^\delta t^{N-1}g(t)\mathrm{d}t\int_0^\delta t^{N-1}h(t)\mathrm{d}t\leq C\int_0^\delta t^{N-1}g(t)h(t)\mathrm{d}t.
\end{align}
\end{lemma}
We are now ready to prove Theorem \ref{imbLip}.
\begin{proof}[Proof of Theorem \ref{imbLip}]
Let $q\in[1,\infty)$ and $(\rho_\eps)_{\eps\in(0,1]}$ be the sequence defined by
\begin{align}
\rho_\eps(t):=\frac{1}{|B_1|}\,\frac{\eps^{1-\eps}}{t^{N-\eps}}\,\mathds{1}_{(0,\eps)}(t) \qquad{\mbox{ for all }}~~\eps\in(0,1]~~\text{ and all }~~t\geq0.
\end{align}
Note that
\begin{align}
\int_0^\infty\rho_\eps(t)\,t^{N-1}\mathrm{d}t=\frac{1}{|B_1|} \qquad{\mbox{ for all }}~~\eps\in(0,1]. \label{masse-lipschitz}
\end{align}
In addition, we also set
\begin{align}
\eta_\eps(h):=\eps^{-N}\,C_2\,\frac{|h|}{\eps}\,\mathds{1}_{B_\eps}(h) \qquad{\mbox{ for all }}~~\eps\in(0,1]~~\text{ and all }~~h\in\R^N,
\end{align}
where $C_2>0$ is a constant such that $\eta_\eps$ has unit mass for each $\eps$.
Notice that $(\eta_\eps)_{\eps>0}\subset L^1(\R^N)$ is a sequence of radial functions satisfying \eqref{molli} and \eqref{molli3}. In particular, by Theorem \ref{THEO2} we know that
\begin{align}
[f]_{C^{0,1}(\R^N)}\lesssim \limsup_{\eps\downarrow0}\int_{B_\eps}\eta_\eps(h)\frac{\|\Delta_h^1f\|_{L^\infty(\R^N)}}{|h|}\mathrm{d}h. \label{frak}
\end{align}
Next, for every $t>0$, define
\begin{align}
F(t):=\int_{\S^{N-1}}\|\Delta_{\sigma t}^1f\|_{L^\infty(\mathbb{R}^N)}\mathrm{d}\mathcal{H}^{N-1}(\sigma).
\end{align}
By the triangle inequality we have $F(2t)\leq 2F(t)$ so that if we let
\[g(t):=\frac{F(t)}{t},\]
we have $g(t)\leq g(t/2)$. In these notations, we have the identity:
\begin{align}
\int_{B_\eps}\rho_\eps(|h|)\frac{\|\Delta_h^1f\|_{L^\infty(\mathbb{R}^N)}}{|h|}\mathrm{d}h=\int_0^\eps t^{N-1}\rho_\eps(t)g(t)\mathrm{d}t.
\end{align}
\noindent Invoking Lemma \ref{chebychev} and \eqref{masse-lipschitz} we deduce that, for every $\eps\in(0,1]$,
\begin{align}
\int_{B_\eps}\rho_\eps(|h|)\frac{\|\Delta_h^1f\|_{L^\infty(\R^N)}}{|h|}\mathrm{d}h&\gtrsim\eps^{-N}\int_0^\eps t^{N-1}\rho_\eps(t)\mathrm{d}t\int_0^\eps t^{N-1}g(t)\mathrm{d}t \nonumber \\
&=\fint_{B_\eps}\frac{\|\Delta_h^1f\|_{L^\infty(\R^N)}}{|h|}\mathrm{d}h \nonumber \\
&\geq \eps^{-1}\fint_{B_\eps}\|\Delta_h^1f\|_{L^\infty(\R^N)}\mathrm{d}h \nonumber \\
&\gtrsim\int_{B_\eps}\eta_\eps(h)\frac{\|\Delta_h^1f\|_{L^\infty(\R^N)}}{|h|}\mathrm{d}h.
\end{align}
Whence, using \eqref{frak} we come up with
\begin{align}
[f]_{C^{0,1}(\R^N)}\lesssim\limsup_{\eps\downarrow0}\int_{B_\eps}\rho_\eps(|h|)\frac{\|\Delta_h^1f\|_{L^\infty(\mathbb{R}^N)}}{|h|}\mathrm{d}h. \label{liiii}
\end{align}
Now, use the Jensen inequality to deduce that
\begin{align}
[f]_{C^{0,1}(\R^N)}^q&\lesssim\limsup_{\eps\downarrow0}\int_{B_\eps}\rho_\eps(|h|)\frac{\|\Delta_h^1f\|_{L^\infty(\mathbb{R}^N)}^q}{|h|^q}\mathrm{d}h \nonumber \\
&\lesssim \limsup_{\eps\downarrow0}~\eps^{-\eps}\left(\eps\int_{B_\eps}\frac{\|\Delta_h^1f\|_{L^\infty(\R^N)}^q}{|h|^{N+q-\eps}}\mathrm{d}h\right) \nonumber \\
&\lesssim \limsup_{\eps\downarrow0}~\eps\int_{\R^N}\frac{\|\Delta_h^1f\|_{L^\infty(\R^N)}^q}{|h|^{N+q-\eps}}\mathrm{d}h. \label{ghjk}
\end{align}
Define $\sigma\in(1-\frac{1}{q},1)$ by the relation $\eps=q(1-\sigma)$. Then,
\begin{align}
[f]_{C^{0,1}(\R^N)}^q&\lesssim \limsup_{\sigma\uparrow1}~q(1-\sigma)[f]_{B_{\infty,q}^\sigma(\R^N)}^q.
\end{align}
The converse of this is covered by Proposition \ref{cotefacile}.
\end{proof}

\section{A non-limiting embedding theorem}\label{sectionLIMITING}

This section is devoted to the proof of Theorem \ref{nonlimit}. The idea of the proof is very similar to that of Theorem 4.4 (ii) on p.36 in \cite{Triebel} (see in particular pp.39-40 there). Nevertheless, we choose to give more details in order to make the dependence of the constants involved on $s$, $p$ and $q$ as explicit as possible.

We will need some preliminary estimates.
\begin{lemma}\label{unboundedcesaro}
Let $(u_j)_{j\geq0}$ be the sequence defined by
\begin{align}
u_j:=\left\{
\begin{array}{r l}
k & \text{if}~~j=2^k~~\text{ for some }~~k\in\N,  \vspace{3pt}\\
0 & \text{else}.
\end{array}
\right.
\end{align}
Then, $(u_j)_{j\geq0}\notin\ell^\infty(\N)$ and
\begin{align}
\sup_{\eps>0}\,\eps\sum_{j\geq0}2^{-j\eps}u_j\leq \frac{2}{e\ln(2)}.
\end{align}
\end{lemma}
\begin{proof}
Let $\eps>0$ and set
\begin{align}
A_\eps:=\sum_{j\geq0}\eps2^{-j\eps}u_j=\sum_{k\geq0}\eps2^{-2^k\eps}k.
\end{align}
Using the (trivial) estimate $e^{-x}\leq1/(ex)$, we have $2^{-x}\leq 1/(ex\ln(2))$. Thus,
$$ A_\eps\leq \frac{1}{e\ln(2)}\sum_{k\geq0}k2^{-k}. $$
Recalling the well-known identity $\sum kx^k=x/(1-x)^{2}$ (for $0\leq x<1$), we finally obtain
$$ A_\eps\leq \frac{2}{e\ln(2)}. $$
Since this holds for every $\eps>0$, we obtain the desired claim.
\end{proof}
\begin{lemma}\label{abovv}
Let $M\in\mathbb{N}^\ast$ and $(u_k)_{k\geq1}$ be a sequence of non-negative numbers. Let $\psi\in C_c^\infty(\mathbb{R})$ be such that $\psi$ is not a polynomial of degree less than or equal to $M-1$, and such that
\[\mathrm{supp}(\psi)\subset [-\eta,\eta] \quad \text{ for some } \quad \eta\geq1,\]
and set
\[f(x_1,...,x_N)=\sum_{k\geq1}u_k\,\psi\left(\frac{x_1-2(M+\eta)k}{2^{-k}}\right)...\,\,\psi\left(\frac{x_N-2(M+\eta)k}{2^{-k}}\right).\]
Then, for any fixed $j\geq1$ we have
\[\sup_{\frac{1}{2^{j+1}}\leq|h|\leq \frac{1}{2^{j}}}\|\Delta_h^Mf\|_{L^p(\mathbb{R}^N)}\geq c\,u_j2^{-j\frac{N}{p}},\]
for some constant $c>0$ depending only on $N$, $p$, $M$ and $\psi$.
\end{lemma}
\begin{proof}
We begin with the case $N=1$. Fix any $j\geq1$ and let $|h|\leq 2^{-j}$. Let us set
$$ x_j:=2(M+\eta)j \quad \text{ and } \quad R_j:=2^{-j}(M+\eta). $$
Since $\mathrm{supp}(\psi)\subset [-\eta,\eta]$ and $|h|\leq2^{-j}$, for any $\ell\in[\![0,M]\!]$, we have
\begin{align}
x\in\mathrm{supp}~\psi\left(\frac{\cdot+ \ell h-x_j}{2^{-j}}\right)&\Leftrightarrow x+h\ell\in [x_j-\eta2^{-j},x_j+\eta2^{-j}] \nonumber \\
&\Leftrightarrow x\in [x_j-h\ell-\eta2^{-j},x_j-h\ell+\eta2^{-j}]=:B_{\ell,j}. \nonumber
\end{align}
And, clearly
$$ \mathrm{supp}\left(\Delta_h^M\psi\left(\frac{\cdot-x_j}{2^{-j}}\right)\right)\subset \bigcup_{\ell\in[\![0,M]\!]}B_{\ell,j}. $$
Thus,
$$ \mathrm{supp}\left(\Delta_h^M\psi\left(\frac{\cdot-x_j}{2^{-j}}\right)\right)\subset [x_j-R_j,x_j+R_j]=:\mathcal{B}_j. $$
Furthermore,
$$ R_{j+1}+R_j=2^{-j}(M+\eta)\bigg(1+\frac{1}{2}\bigg)<2(M+\eta)=x_{j+1}-x_j, $$
and so, the $\mathcal{B}_j$'s are mutually disjoint. Therefore, given any fixed $j\geq1$ and $\varepsilon>0$ a small parameter less than $R_j$, we have
\begin{align}
\|\Delta_h^Mf\|_{L^p(\mathbb{R})}^p&=\left\|\sum_{k\geq1}u_k\Delta_h^M\psi\left(\frac{\cdot-x_k}{2^{-k}}\right)\right\|_{L^p(\mathbb{R})}^p \nonumber \\
&\geq \int_{x_j-\varepsilon}^{x_j+\varepsilon}\bigg|\sum_{k\geq1}u_k\Delta_h^M\psi\left(\frac{x-x_k}{2^{-k}}\right)\bigg|^p\mathrm{d}x \nonumber \\
&= u_j^p\int_{x_j-\varepsilon}^{x_j+\varepsilon}\left|\Delta_h^M\psi\left(\frac{x-x_j}{2^{-j}}\right)\right|^p\mathrm{d}x \nonumber \\
&= u_j^p2^{-j}\int_{-\varepsilon/2^{-j}}^{\varepsilon/2^{-j}}|\Delta_{h/2^{-j}}^M\psi(x)|^p\mathrm{d}x.  \label{KonkLemma}
\end{align}
Whence, writing $K_j:=\overline{B_{2^{-j}}\setminus B_{2^{-(j+1)}}}$ for $j\geq0$ we have
\begin{align}
\sup_{h\in K_j}\|\Delta_h^Mf\|_{L^p(\mathbb{R})}^p&\geq u_j^p2^{-j}\sup_{h\in K_j}\int_{-\varepsilon}^{\varepsilon}|\Delta_{h/2^{-j}}^M\psi(x)|^p\mathrm{d}x \nonumber \\
&=c_\eps^p\,u_j^p2^{-j}.
\end{align}
where
\begin{align}
c_\eps=c_\eps(M,p,\psi):=\sup_{\frac{1}{2}\leq|h|\leq1}\left(\int_{-\varepsilon}^{\varepsilon}|\Delta_{h}^M\psi(x)|^p\mathrm{d}x\right)^{1/p}.
\end{align}
Since $\varepsilon>0$ is an arbitrary small parameter and $\psi$ is not a polynomial of degree less than or equal to $M-1$, we may find a number $\eps_0>0$ such that $c_{\eps_0}>0$.

The proof when $N\geq2$ follows by a straightforward adaptation of the case $N=1$ using the product structure $\psi(x_1)...\psi(x_N)$ and Fubini's theorem which gives the result with $c=c_{\eps_0}^N$.
\end{proof}

We are now ready to prove Theorem \ref{nonlimit}.
\begin{proof}[Proof of Theorem \ref{nonlimit}]
Let $M\in\N^\ast$ such that $s\in(0,M)$ and let $u_j$ be the sequence of Lemma \ref{unboundedcesaro}. Also, we let
$\psi\in C_c^\infty(\R^N)$ be such that
\begin{align}
\mathrm{supp}(\psi)\subset B_2 \quad \text{  and  } \quad \sum_{m\in\Z^N}\psi(x-m)=1, \quad \text{for any}~~x\in\R^N.
\end{align}
In addition, we suppose that $\psi$ has the product structure
$$ \psi(x)=\Psi(x_1)\,...\,\Psi(x_N), $$
for some $\Psi\in C_c^\infty(\R)$ different from a polynomial of degree less than or equal to $M-1$.
Then, we set
$$ m_j:=2(M+2)j\,\mathbf{1}\in\Z^N \qquad{\mbox{ with }}~~\mathbf{1}:=(1,...,1)\in\Z^N, $$
and we define
\begin{align}
f(x)&:=\sum_{j\geq1}u_j^{1/q}2^{-j(s-\frac{N}{p})}\psi(2^j(x-m_j)) \nonumber \\
&=\sum_{j\geq1}\big(u_j^{1/q}2^{-j\eps}\big)2^{-j(s-\eps-\frac{N}{p})}\psi(2^j(x-m_j)).
\end{align}
where $x\in\R^N$. It follows from Definition \ref{quarks} that
\begin{align}
2^{-j(s-\eps-\frac{N}{p})}\psi(2^j(x-m_j))
\end{align}
can be interpreted as $(s-\eps,p)$-$0$-quarks.
Accordingly, by Definition \ref{DEFquarkonial} we have that
\begin{align}
\eps\|f\|_{\mathbf{B}_{p,q}^{s-\eps}(\R^N)}^q\leq \eps\sum_{j\geq1}\big(2^{-j\eps}u_j^{1/q}\big)^{q}.
\end{align}
Using Lemma \ref{unboundedcesaro} we obtain that
\begin{align}
\eps\|f\|_{\mathbf{B}_{p,q}^{s-\eps}(\R^N)}^q\leq \frac{2q^{-1}}{e \ln(2)}<\infty, \quad \forall\eps\in(0,s).
\end{align}
In particular, recalling Theorem \ref{THquark}, $f\in L^p(\R^N)$. Also, for all $j\geq1$, we write
\begin{align}
K_j:=\{2^{-(j+1)}\leq|h|\leq2^{-j}\}.
\end{align}
Recall that
\begin{align}
\|f\|_{B_{p,\infty}^s(\R^N)}\sim \|f\|_{L^p(\R^N)}+\sup_{j\geq1}~2^{js}\sup_{h\in K_j}\|\Delta_h^Mf\|_{L^p(\R^N)},
\end{align}
is an equivalent norm on $B_{p,\infty}^s(\R^N)$ (this is a discretized version of Theorem 2.5.12 on p.110 in \cite{Triebel}).
Using this together with Lemma \ref{abovv} we get
\begin{align}
\|f\|_{B_{p,\infty}^s(\R^N)}\geq c\,\sup_{j\geq1}u_j^{1/q}=\infty.
\end{align}
Here $c=c(N,p,M,\psi)>0$. Thus $f\notin B_{p,\infty}^s(\R^N)$.
This completes the proof.
\end{proof}

\section{Non-compactness results}\label{sectionCOMPACT}

This section is devoted to the proofs of Theorem \ref{noncompactness} and Theorem \ref{noncpctB}. We begin with the former one.

\begin{proof}[Proof of Theorem \ref{noncompactness}]
For simplicity, we replace $\eps>0$ by $1/n$ with $n\geq1$ and write $\rho_n$ instead of $\rho_{1/n}$.
We write
\begin{align}
x&=(x_1,...,x_N)\in\R^N, \nonumber \\
y&=(x_1,...,x_{N-1})\in \R^{N-1}, \nonumber
\end{align}
and, for all $n\geq1$, we let
\[f_n(x):=n^{\frac{M-s}{Mp}}\Phi(n^{\frac{M-s}{M}}x_N)\varphi(y),\]
for some arbitrary $\Phi\in C_c^\infty(\R)$ and $\varphi\in C_c^\infty(\R^{N-1})$ (if $N=1$, replace $\varphi$ by $1$) with
\begin{align}
\max\{\|\Phi\|_{W^{M,p}(\R)},\|\varphi\|_{W^{M,p}(\R^{N-1})}\}\leq C_0. \label{hypppp}
\end{align}
Note that
\begin{align}
f_n\to 0~~\text{  a.e. in  }~~\R^N.
\end{align}
Further, from Fubini's theorem we infer that
\begin{align}
\int_{\R^N}|f_n(x)|^p\mathrm{d}x&=\left(\int_{\R^{N-1}}|\varphi(y)|^p\mathrm{d}y\right)\left(n^{\frac{M-s}{M}}\int_{\R}|\Phi(n^{\frac{M-s}{M}}x_N)|^p\mathrm{d}x_N\right) \nonumber \\
&=\|\varphi\|_{L^p(\R^{N-1})}^p\|\Phi\|_{L^p(\R)}^p= C_1. \nonumber
\end{align}
On the one hand, we observe that \eqref{hypppp} gives
\begin{align}
\|D^\alpha f_n\|_{L^p(\R^N)}^p\lesssim 1~~\text{  for each  }~~\alpha=(\alpha_1,...,\alpha_{N-1},0)\in\N^N~~\text{  with  }~~|\alpha|\leq M.
\end{align}
While, on the other hand, for all $j\in[\![1,M]\!]$,
\begin{align}
\left\|\frac{\partial^jf_n}{\partial x_N^j}\right\|_{L^p(\R^N)}^p&=n^{\frac{(M-s)}{M}jp}\|\varphi\|_{L^p(\R^{N-1})}^p~n^{\frac{M-s}{M}}\int_{\R}|\Phi^{(j)}(n^{\frac{M-s}{M}}x_N)|^p\mathrm{d}x_N \nonumber \\
&=n^{\frac{(M-s)}{M}jp}\|\varphi\|_{L^p(\R^{N-1})}^p\|\Phi^{(j)}\|_{L^p(\R)}^p \nonumber \\
&\leq n^{(M-s)p}\|\varphi\|_{L^p(\R^{N-1})}^p~\max_{j\in[\![1,M]\!]}\|\Phi^{(j)}\|_{L^p(\R)}^p.
\end{align}
Whence, using the product structure of $f_n$ we get
\begin{align}
\sup_{|\alpha|\leq M}\|D^\alpha f_n\|_{L^p(\R^N)}^p\leq C_2n^{(M-s)p}, \quad \text{ for all }~~n\geq1.
\end{align}
Moreover, for all $h\ne0$, it holds
\begin{align}
\|\Delta_h^Mf_n\|_{L^p(\R^N)}^p&\leq |h|^{Mp}\sup_{|\alpha|\leq M}\|D^\alpha f_n\|_{L^p(\R^N)}^p \nonumber \\
&\leq C_2|h|^{Mp}n^{(M-s)p}.
\end{align}
Then,
\begin{align}
\int_{\R^N}\int_{\R^N}\rho_n(h)\frac{|\Delta_h^Mf_n(x)|^p}{|h|^{sp}}\mathrm{d}x\mathrm{d}h&\lesssim \int_{\R^N}\rho_n(h)|h|^{p(M-s)}n^{p(M-s)}\mathrm{d}h \nonumber \\
&=\int_{\R^N}\rho(h)|h|^{p(M-s)}\mathrm{d}h.
\end{align}
We thus conclude that
\begin{align}
\int_{\R^N}\int_{\R^N}\rho_n(h)\frac{|\Delta_h^Mf_n(x)|^p}{|h|^{sp}}\mathrm{d}x\mathrm{d}h\leq C_3 \quad \text{  for any  }~~n\geq1.
\end{align}
Yet, $(f_n)_{n\geq1}$ is not relatively compact in $L_{loc}^p(\R^N)$.
\end{proof}

The proof of Theorem \ref{noncpctB} is as follows.

\begin{proof}[Proof of Theorem \ref{noncpctB}]

The proof in this case is very similar to that of Theorem \ref{noncompactness}.
We let $M\in\N^\ast$, $s\in(0,M)$ and pick a slightly different sequence of functions, for example
\begin{align}
f_n(x):=n^{\frac{\gamma}{Mp}}\Phi(n^{\frac{\gamma}{M}}x_N)\varphi(y),
\end{align}
where $0\leq \gamma\leq\frac{1}{q}$, $\Phi\in C_c^\infty(\R)$ and $\varphi\in C_c^\infty(\R^{N-1})$. Also, we set
\begin{align}
\rho_n(h):=\frac{1}{n\sigma_N|h|^{N-1/n}}\mathds{1}_{(0,1)}(|h|).
\end{align}
As before,
\begin{align}
\|f_n\|_{L^p(\R^N)}=\|\Phi\|_{L^p(\R)}\|\varphi\|_{L^p(\R^{N-1})}.
\end{align}
And
\begin{align}
\sup_{|\alpha|\leq M}\|D^\alpha f_n\|_{L^p(\R^N)}\lesssim n^{\gamma}.
\end{align}
Whence,
\begin{align}
\int_{B_1}\rho_n(h)\frac{\|\Delta_h^Mf_n\|_{L^p(\R^N)}^q}{|h|^{sq}}\mathrm{d}h&\lesssim \int_{B_1}\rho_n(h)|h|^{(M-s)q}n^{\gamma q}\mathrm{d}h \nonumber \\
&=\int_0^1n^{\gamma q-1}\frac{\mathrm{d} r}{r^{1-(q(M-s)+1/n)}} \nonumber \\
&=\frac{n^{\gamma q}}{1+(M-s)qn}\lesssim \frac{1}{n^{1-\gamma q}}.
\end{align}
Since $0\leq \gamma\leq\frac{1}{q}$ we obtain
\begin{align}
\int_{B_1}\rho_n(h)\frac{\|\Delta_h^Mf_n\|_{L^p(\R^N)}^q}{|h|^{sq}}\mathrm{d}h\leq C \quad \text{ for all }~~n\geq1.
\end{align}
However, $(f_n)_{n\geq1}$ is not relatively compact in $L_{loc}^p(\R^N)$.
\end{proof}

\section*{Appendix}

In \cite{LaMi}, Lamy and Mironescu proved the
\begin{theorem}[Lamy, Mironescu, \cite{LaMi}]\label{lami}
Let $s>0$, $p\in[1,\infty)$ and $(\rho_\varepsilon)_{\varepsilon>0}$ satisfying \eqref{molli} and \eqref{molli3}. Then,
\begin{align}
\|f\|_{B_{p\infty}^s(\R^N)}\lesssim \|f\|_{L^p(\R^N)}+\sup_{\varepsilon\in(0,1)}\frac{\|\rho_\varepsilon\ast f-f\|_{L^p(\R^N)}}{\varepsilon^s}. \label{eqLM}
\end{align}
\end{theorem}
Since Theorem \ref{lami} is not properly stated in \cite{LaMi} nor its proof, we shall give a brief sketch of the proof in order to justify that their result indeed applies to the scale $B_{p,\infty}^s(\R^N)$.
\begin{proof}[Sketch of the proof]
It is well-known that each tempered distribution $f\in\mathscr{S}'(\R^N)$ can be decomposed as
\begin{align}
f=\sum_{j\geq0}f_j, \label{LP}
\end{align}
where $f_0=f\ast\zeta$, $~f_j=f\ast\varphi_{2^{1-j}}$, $j\geq1$, and $\zeta,\varphi\in\mathscr{S}(\R^N)$ are functions satisfying
\begin{enumerate}
\item[(i)] $\mathrm{supp}(\hat{\zeta})\subset B_2$ and $\hat{\zeta}\equiv1$ in a neighborhood of $\bar{B}_1$,
\item[(ii)] $\varphi:=\zeta_{1/2}-\zeta$ with $\hat{\varphi}=\hat{\zeta}(\cdot/2)-\hat{\zeta}$ and $\mathrm{supp}(\hat{\varphi})\subset B_4\setminus\bar{B}_1$.
\end{enumerate}
where the subscript $\varphi_k$ means $k^{-N}\varphi(\cdot/k)$ and $\hat{\varphi}$ stands for the Fourier transform of $\varphi$ (similarly for $\zeta$). Formula \eqref{LP} is called the Littlewood-Paley decomposition of $f$. Furthermore, it is known that each function in $B_{p,\infty}^s(\R^N)$ is a tempered distribution, so that this decomposition makes sense here and may even serve to formulate an equivalent norm on this space via the formula
\[\|f\|_{B_{p,\infty}^s(\R^N)}\sim \sup_{j\geq0}2^{js}\|f_j\|_{L^p(\R^N)}.\]
To see that Theorem \ref{lami} holds it suffices to discretize the last term on the right-hand side of \eqref{eqLM} as
\[\sup_{\eps\in(1/2,1)}\sup_{j\geq0}2^{js}\|f-f\ast\rho_{2^{-j}\eps}\|_{L^p(\R^N)}.\]
At this stage, all the estimates obtained in \cite{LaMi} directly apply because it is the terms $\|f_j\|_{L^p(\R^N)}$ which are estimated there (and not their sum nor their integral) in terms of the quantity $\|f-f\ast\rho_{2^{-j}\eps}\|_{L^p(\R^N)}$.
\end{proof}
Using this result, Proposition \ref{nonextension} can be proved by arguing as follows.
\begin{proof}[Proof of Proposition \ref{nonextension}]
Suppose without loss of generality that the $\rho_\eps$'s are compactly supported and that $\mathrm{supp}(\rho)\subset B_{1}$. Also, up to replace $\rho_\eps$ by $\frac{\rho_\eps(h)+\rho_\eps(-h)}{2}$, we can always assume that each $\rho_\eps$ is even. Then, by the Jensen inequality,
\begin{align}
\frac{\|\rho_\varepsilon\ast f-f\|_{L^p(\R^N)}^p}{\varepsilon^{sp}}&=\frac{1}{\varepsilon^{sp}}\int_{\R^N}\left|\int_{\R^N}\rho_\varepsilon(h)[f(x-h)-f(x)]\mathrm{d}h\right|^p\mathrm{d}x \nonumber \\
&\leq \int_{\R^N}\int_{B_{\eps}}\rho_\varepsilon(-h)\frac{|f(x+h)-f(x)|^p}{\varepsilon^{sp}}\mathrm{d}h\mathrm{d}x \nonumber \\
&\leq \int_{\R^N}\int_{B_{\eps}}\rho_\varepsilon(h)\frac{|f(x+h)-f(x)|^p}{|h|^{sp}}\mathrm{d}h\mathrm{d}x. \nonumber
\end{align}
Whence,
\[\sup_{\varepsilon\in(0,1)}\frac{\|\rho_\varepsilon\ast f-f\|_{L^p(\R^N)}^p}{\varepsilon^{sp}}\lesssim \sup_{\varepsilon\in(0,1)}\int_{\R^N}\int_{\R^N}\rho_\varepsilon(h)\frac{|f(x+h)-f(x)|^p}{|h|^{sp}}\mathrm{d}h\mathrm{d}x.\]
And so, by Theorem \ref{lami}, $f\in B_{p,\infty}^s(\R^N)$. The proof when $\rho$ is not compactly supported follows by a simple comparison argument: cutting off $\rho$ as $\tilde{\rho}:=\rho\mathds{1}_{B_R}$ for some $R>0$ with $|B_R\cap\mathrm{supp}(\rho)|>0$, we clearly have $\rho\geq\tilde{\rho}$ and \eqref{kpp} implies that the same property holds for $\tilde{\rho}$ instead of $\rho$ (up to some multiplicative factor $\|\rho\|_{L^1(B_R)}$ to make $\tilde{\rho}_\eps$ a sequence of mollifiers), i.e. that $f\in B_{p,\infty}^s(\R^N)$.
\end{proof}

%Note that, in fact, we have proved the following
%\begin{theorem}
%Let $f:\R^N\to\R_+$ be such that
%\begin{align}
%f(h)\leq C\left\{f(z)+f(z-h)\right\}~~\text{for any}~~h\in\R^N,~~z\in B_{|h|}(h).
%\end{align}
%In addition, let $(\rho_\eps)_{\eps>0}\subset L^1(\R^N)$ be a sequence of radial functions satisfying \eqref{molli} and \eqref{molli3}. Then,
%\begin{align}
%\limsup_{|x|\to0}f(x)\sim\limsup_{\eps\downarrow0}\int_{\R^N}\rho_\eps(x)f(x)\mathrm{d}x.
%\end{align}
%\end{theorem}

\section*{Acknowledgments} The author would like to express his gratitude to \emph{J\'er\^ome Coville} for suggesting him the problem and for careful reading of the manuscript. The author is grateful to \emph{Fran\c{c}ois Hamel} who made valuable comments, and to \emph{Petru Mironescu} whose lessons have been an inspiration for this work. The author warmly thanks the anonymous referees whose insightful comments helped improve and clarify this manuscript. This project has been supported by the French National Research Agency (ANR) in the framework of the ANR NONLOCAL project (ANR-14-CE25-0013).

\vspace{2mm}

\end{document}